\documentclass[11pt,english]{article}
\usepackage[utf8]{inputenc} 
\usepackage[T1]{fontenc}
\usepackage[margin=2.00 cm,bottom=20mm,footskip=7mm,top=20mm]{geometry}

\usepackage{amsthm}
\usepackage{amsmath}
\usepackage{amssymb}
\usepackage{setspace}
\usepackage{mathtools}
\usepackage{graphicx}
\usepackage[hidelinks]{hyperref}
\usepackage{thm-restate}
\usepackage{cleveref}
\usepackage{graphicx}
\usepackage{enumitem}
\usepackage{makecell}
\usepackage{mdframed}
\usepackage{comment}
\usepackage{framed}
\usepackage{subcaption}

\usepackage{floatrow}

\usepackage{tikz}
\usepackage{mathdots}
\usetikzlibrary{calc}
\usetikzlibrary{decorations.pathreplacing}
\usetikzlibrary{positioning,patterns}
\usetikzlibrary{arrows,shapes,positioning}
\usetikzlibrary{decorations.markings}
\usepackage[utf8]{inputenc}

\def \vx {circle[radius = .15][fill = white]}
\def \bvx {circle[radius = .15][fill = black]}

\tikzstyle{edge}=[very thick]
\tikzstyle{diredge}=[postaction={decorate,decoration={markings,
		mark=at position .55 with {\arrow[scale = 1]{stealth};}}}]
\newcommand{\defPt}[3]{
	\def \pt {(#1, #2)}
	\coordinate [at = \pt, name = #3];
}

\AtBeginDocument{%
   \def\MR#1{}
}

\usepackage[square,sort,comma,numbers]{natbib}
\theoremstyle{plain}

\newtheorem*{thm*}{Theorem}
\newtheorem{thm}{Theorem}
\Crefname{thm}{Theorem}{Theorems}
\numberwithin{thm}{section}

\newtheorem*{lem*}{Lemma}
\newtheorem{lem}[thm]{Lemma}
\Crefname{lem}{Lemma}{Lemmas}

\newtheorem*{claim*}{Claim}

\crefname{claim}{Claim}{Claims}
\Crefname{claim}{Claim}{Claims}

\newtheorem*{defn*}{Definition}
\Crefname{defn*}{Definition}{Definitions}

\Crefname{prop}{Proposition}{Propositions}

\crefname{cor}{Corollary}{Corollaries}

\crefname{conj}{Conjecture}{Conjectures}

\newtheorem{qn}[thm]{Question}
\Crefname{qn}{Question}{Questions}

\newtheorem{obs}[thm]{Observation}
\Crefname{obs}{Observation}{Observations}

\Crefname{ex}{Example}{Examples}

\theoremstyle{definition}

\Crefname{prob}{Problem}{Problems}

\newtheorem{defn}[thm]{Definition}
\Crefname{defn}{Definition}{Definitions}

\theoremstyle{remark}

\renewenvironment{proof}[1][]{\begin{trivlist}
\item[\hspace{\labelsep}{\bf\noindent Proof#1.\/}] }{\qed\end{trivlist}}

\newcommand{\remove}[1]{}

\newcommand{\floor}[1]{
    \left \lfloor #1 \right \rfloor
}

\expandafter\def\expandafter\normalsize\expandafter{%
    \normalsize
    \setlength\abovedisplayskip{4pt}
    \setlength\belowdisplayskip{4pt}
    \setlength\abovedisplayshortskip{4pt}
    \setlength\belowdisplayshortskip{4pt}
}

\newmdenv[
  topline=false,
  bottomline=false,
  skipabove=\topsep,
  skipbelow=\topsep
]{siderules}

\DeclareMathOperator{\LP}{LP} 
\DeclareMathOperator{\LMS}{LMS} 
\DeclareMathOperator{\towr}{towr}
\newcommand{\ML}{F}
\DeclareMathOperator{\bR}{\mathbb{R}}

\newcommand{\bZ}{\mathbb{Z}}

\newcommand{\N}{\mathbb{N}}
\newcommand{\R}{\mathbb{R}}
\newcommand{\vect}[1]{\boldsymbol{#1}}
\usepackage{stmaryrd}

\title{Erd\H{o}s-Szekeres theorem for multidimensional arrays}
\date{}
\author{
Matija Buci\'c\thanks{School of Mathematics, Institute for Advanced Study and Department of Mathematics, Princeton University, 304 Washington Road, 08544 NJ, USA. Email: \href{mailto:matija.bucic@ias.edu} {\nolinkurl{matija.bucic@ias.edu}}.}
\and
Benny Sudakov\thanks{Department of Mathematics, ETH, Z\"urich, Switzerland. Email:
\href{mailto:benjamin.sudakov@math.ethz.ch} {\nolinkurl{benjamin.sudakov@math.ethz.ch}}.
Research supported in part by SNSF grant 200021\_196965.}
\and
Tuan Tran\thanks{Discrete Mathematics Group, Institute for Basic Science (IBS), Daejeon, Republic of Korea. Email:
\href{mailto:tuantran@ibs.re.kr} {\nolinkurl{tuantran@ibs.re.kr}}.
This work was supported by the Institute for Basic Science (IBS-R029-Y1).}
}
\begin{document}

\maketitle
\vspace{-0.5cm}
\begin{abstract}
The classical Erd\H{o}s-Szekeres theorem dating back almost a hundred years states that any sequence of $(n-1)^2+1$ distinct real numbers contains a monotone subsequence of length $n$. This theorem has been  
generalised to higher dimensions in a variety of ways
but perhaps the most natural one was proposed by Fishburn and Graham more than 25 years ago. They 
defined the concept of a monotone and a lex-monotone array and 
asked how large an array one needs in order to be able to find a monotone or a lex-monotone subarray of size $n \times \ldots \times n$. Fishburn and Graham 
obtained Ackerman-type bounds in both cases. We significantly improve these results. Regardless of the dimension we obtain at most a triple exponential bound in $n$ in the monotone case and a quadruple exponential one in the lex-monotone case. 
\end{abstract}

\section{Introduction}

A classical paper of Erd\H{o}s and Szekeres \cite{ES35} from 1935 is one of the starting points of a very rich discipline within combinatorics: Ramsey theory. A main result of the paper, which has become known as the Erd\H{o}s-Szekeres theorem, says that any sequence of $(n-1)^2+1$ distinct real numbers contains either an increasing or decreasing subsequence of length $n$, and this is tight. Among simple results in combinatorics, only few can compete with this one in terms of beauty and utility. See, for example, Steele \cite{Ste95} for a collection of six proofs and some applications.

A very natural question which arises is how does one
generalise the Erd\H{o}s-Szekeres theorem to higher dimensions? The main concept which does not have an obvious generalisation is that of the monotonicity of a subsequence. Several candidates have been proposed \cite{Kru53,Mor69,BM73-I,BM73-II,Kal73,Sid99,ST01, LS18} but perhaps the most natural one was introduced more than 25 years ago by Fishburn and Graham \cite{FG93}. A multidimensional array is said to be monotone if for each dimension all the 1-dimensional subarrays along the direction of this dimension are increasing or are all decreasing. To be more formal, a $d$-dimensional array $f$ is an injective function from $A_1\times \ldots \times A_d$ to $\mathbb{R}$ where $A_1,\ldots,A_d$ are non-empty subsets of $\mathbb{Z}$; we say $f$ has size $|A_1| \times \ldots \times |A_d|$.

\begin{defn*}[Monotone array]
 A $d$-dimensional array $f\colon A_1\times \ldots \times A_d \rightarrow \mathbb{R}$ is monotone if for each $i\in [d]$ one of the following alternatives occurs:
\begin{itemize}
    \item[\rm (i)] $f(a_1,\ldots,a_{i-1},x,a_{i+1},\ldots,a_d)$ is increasing in $x$ for all choices of $a_1,\ldots, a_{i-1},a_{i+1},\ldots, a_d$;
    \item[\rm (ii)] $f(a_1,\ldots,a_{i-1},x,a_{i+1},\ldots,a_d)$ is decreasing in $x$ for all choices of $a_1,\ldots, a_{i-1},a_{i+1},\ldots, a_d$.
\end{itemize}
\end{defn*}
For example, of the following 2-dimensional arrays first and second are monotone, while the third is not (since some rows contain increasing and some rows decreasing sequences).  

\begin{center}
\begin{tabular}{ | c | c | c | c c | c | c | c |c c | c | c | c |}
 \cline{1-3}\cline{6-8} \cline{11-13}
 7 & 8 & 9 & & & 1 & 3 & 6 & & & 7 & 8 & 9\\ \cline{1-3}\cline{6-8}\cline{11-13} 
 4 & 5 & 6 & & & 2 & 5 & 7 & & & 6 & 5 & 4\\  \cline{1-3}\cline{6-8}\cline{11-13}
 1 & 2 & 3 & & & 4 & 8 & 9 & & & 1 & 2 & 3\\ \cline{1-3}\cline{6-8} \cline{11-13}
\end{tabular}
\end{center}

The higher dimensional version of the Erd\H{o}s-Szekeres problem introduced by Fishburn and Graham \cite{FG93} now becomes: given positive integers $d$ and $n$, determine the smallest $N$ such that any $d$-dimensional array of size $N\times \ldots \times N$ contains a monotone $d$-dimensional subarray of size $n\times \ldots \times n$, we denote this $N$ by $M_d(n)$. The Erd\H{o}s-Szekeres theorem can now be rephrased as $M_1(n)=(n-1)^2+1$. Fishburn and Graham \cite[Section 3]{FG93} showed that $M_2(n) \le \towr_5(O(n))$\footnote{We define the tower function $\towr_k(x)$ by $\towr_1(x)=x$ and $\towr_{k}(x)=2^{\towr_{k-1}(x)}$ for $k\ge 2$.}, that $M_3(n)$ is bounded by a tower of height at least a tower in $n$ and that $M_d(n)$ is bounded from above by an Ackermann-type\footnote{The Ackermann function $A_k$ of order $k$ is defined recursively by $A_k(1)=2, A_1(n)=2n$ and $A_k(n)=A_{k-1}(A_k(n-1))$. It is an incredibly fast growing function, for example $A_2(n)=2^n$, $A_3(n)=\towr_n(2)$ and $A_4(n)$ is a tower of height tower of height tower, iterated $n$ times, of $2$.} function of order at least $d$ for $d\ge 4$.
We significantly improve upon these results. 
\begin{thm}\label{thm:monotone}
\textcolor{white}{ }
\begin{enumerate}[label=(\roman*), ref=(\roman*)]
    \item\label{itm:2d} $M_2(n) \le 2^{2^{(2+o(1))n}}$,
    \item\label{itm:3d} $M_3(n) \le 2^{2^{(2+o(1))n^2}}$ and
     \item\label{itm:4+d} $M_d(n)\le 2^{2^{2^{O_d(n^{d-1})}}}$, for $d \ge 4$. 
\end{enumerate}
\end{thm} 

Fishburn and Graham introduced another very natural generalisation of the notion of monotonicity of a sequence to higher dimensional arrays. A multidimensional array is said to be \textit{lexicographic} if for any two entries the one which has the larger position in the first coordinate in which they differ is larger. For example, the following array is lexicographic:

\begin{center}
\begin{tabular}{ | c | c | c | }
 \cline{1-3}
 3 & 6 & 9 \\ \hline
 2 & 5 & 8 \\ \hline
 1 & 4 & 7 \\ \hline
\end{tabular}
\end{center}

An array is said to be \textit{lex-monotone} if it is possible to permute the coordinates and reflect the array along some dimensions to obtain a lexicographic array. To be more formal, for two vectors $\vect{u}=(u_1,\ldots,u_d)$ and $\vect{v}=(v_1,\ldots,v_d)$ in $\R^d$, we write $\vect{u}<_{\mbox{lex}} \vect{v}$ if $u_i<v_i$, where $i$ is the smallest index such that $u_i\neq v_i$.
\begin{defn*}[Lex-monotone array]
A $d$-dimensional array $f$ is said to be lex-monotone if there exist a permutation $\sigma\colon [d]\rightarrow [d]$ and a sign vector $\vect{s} \in \{-1,1\}^d$ such that 
\[
f(\vect{x})<f(\vect{y}) \Leftrightarrow (s_{\sigma(1)}x_{\sigma(1)},\ldots,s_{\sigma(d)} x_{\sigma(d)})<_{\hbox{lex}} (s_{\sigma(1)}y_{\sigma(1)},\ldots,s_{\sigma(d)}y_{\sigma(d)}).
\]
\end{defn*}
Note that a 1-dimensional array is lex-monotone if and only if it is a monotone sequence. The following 
2-dimensional arrays are lex-monotone since for the first one the above matrix is obtained by swapping the coordinates, for the second one by reflecting along the first dimension and for the third by performing both of these operations.

\begin{center}
\begin{tabular}{ | c | c | c | c c | c | c | c |c c | c | c | c |}
 \cline{1-3}\cline{6-8} \cline{11-13}
 7 & 8 & 9 & & & 9 & 6 & 3 & & & 9 & 8 & 7\\ \cline{1-3}\cline{6-8}\cline{11-13} 
 4 & 5 & 6 & & & 8 & 5 & 2 & & & 6 & 5 & 4\\  \cline{1-3}\cline{6-8}\cline{11-13}
 1 & 2 & 3 & & & 7 & 4 & 1 & & & 3 & 2 & 1\\ \cline{1-3}\cline{6-8} \cline{11-13}
\end{tabular}
\end{center}

Given positive integers $d$ and $n$, let $L_d(n)$ denote the minimum $N$ such that for any $d$-dimensional array of size $N\times \ldots \times N$, one can find a lex-monotone subarray of size $n\times \ldots \times n$. Fishburn and Graham \cite[Theorem 1]{FG93} showed that $L_d(n)$  exists. This result has been used to prove interesting results in poset dimension theory \cite{FFT99} and computational complexity theory \cite{BK10}.

Note that any lex-monotone array is monotone, so a very natural strategy to bound $L_d(n)$ is to first find a monotone subarray and then within this subarray find a lex-monotone subarray. This motivates the following problem which is of independent interest. For positive integers $d$ and $n$, we define $\ML_d(n)$ to be the minimum $N$ such that any $d$-dimensional monotone array of size $N\times \ldots \times N$, contains a lex-monotone subarray of size $n\times \ldots \times n$. It is easy to see by the above reasoning that $L_d(n) \le M_d(\ML_d(n))$. Fishburn and Graham \cite[Lemma 1]{FG93} showed 
$\ML_2(n) \le 2n^2-5n+4$ and $\ML_3(n) \le 2^{2n+o(n)}$, while for $d\ge 4$ their argument gives $\ML_d(n) \le \towr_{d-1}(O_d(n))$. We determine $\ML_2(n)$ completely and significantly improve 
the bound for all $d\ge 3$.
\begin{thm}\label{thm:mon-to-lex}
\textcolor{white}{ }
\begin{enumerate}[label=(\roman*), ref=(\roman*)]
    \item $\ML_2(n) = 2n^2-5n+4$ and
    \item $\ML_d(n) \le 2^{O_d(n^{d-2})}$, for $d \ge 3$.
\end{enumerate}
\end{thm}
Part (i) of Theorem \ref{thm:mon-to-lex} answers a question of Fishburn and Graham asking whether $\ML_2(n)=(1+o(1))n^2$, in negative. 
Combining Theorems \ref{thm:monotone} and \ref{thm:mon-to-lex} with the inequality $L_d(n) \le M_d(\ML_d(n))$ gives the following upper bounds on $L_d(n)$.

\begin{thm}\label{thm:lexicographic}
\textcolor{white}{ }
\begin{enumerate}[label=(\roman*), ref=(\roman*)]
    \item $L_2(n) \le 2^{2^{(4+o(1))n^2}}$,
    \item $L_3(n)\le 2^{2^{2^{(2+o(1))n}}}$ and
    \item $L_d(n) \le \towr_5\left(O_d(n^{d-2})\right)$, for $d \ge 4$.
\end{enumerate}
\end{thm}
For comparison, the best lower bound on $L_d(n)$, due to Fishburn and Graham \cite[Theorem 2]{FG93}, is $L_d(n) \ge n^{(1-1/d)n^{d-1}}$ for all $d\ge 2, n \ge 3$, achieved by taking a random array.  

\vspace{1ex}
\noindent\textbf{Notation and organisation.}
The rest of the paper is organised as follows. We prove  \Cref{thm:monotone} in \Cref{sec:monotone} and \Cref{thm:mon-to-lex} in Section \ref{sec:lexicographic}. The final section contains some concluding remarks and open problems.

We use standard set-theoretic and asymptotic notation throughout the paper. We write $[n]$ for $\{1,2,\ldots,n\}$. For sequences $a(n)$ and $b(n)$ we write $a(n)=O(b(n))$ to mean there is a constant $C$ such that $|a(n)| \le C|b(n)|$ for all $n \in \N$, and we write $a(n)=o(b(n))$ to mean that $a(n)/b(n) \rightarrow 0$ as $n\rightarrow \infty$. For the purpose of asymptotics we always treat $d$, the number of dimensions, as a constant while taking $n\rightarrow \infty$. Given $d\in \N$, we denote the set of all permutations of $[d]$ by $\mathfrak{S}_d$. For real numbers $\alpha$ and $\beta$, we employ the interval notation
\[
[\alpha,\beta]:= \{x \in \bZ: \alpha \le x \le \beta \}.
\]
A set of the form $A_1\times \ldots \times A_d$, where $A_i$ is a finite subset of $\mathbb{Z}$ for each $i \in [d]$, is called an $|A_1|\times \ldots \times |A_d|$ grid or a grid of size $|A_1|\times \ldots \times |A_d|$. Note that a $d$-dimensional array $f\colon A_1\times \ldots \times A_d \rightarrow \bR$ is equivalent to an ordering of the vertices of the $d$-dimensional grid $A_1\times \ldots \times A_d$; we switch between these points of view interchangeably.

We generally use lowercase bolded letters for vectors and uppercase bolded letters for grids. For a vector $\vect{u}$ we denote by $u_i$ the value of the $i$-th coordinate.
\section{Monotone arrays}
\label{sec:monotone}
 
In this section we will prove \Cref{thm:monotone}.
We begin with a few preliminaries.

\subsection{Preliminaries}\label{subsec:prelim}
We collect here two well-known Ramsey-type results for grids. We mention first a relation between the grid Ramsey problem and the (hyper)graph Zarankiewicz problem which offers an alternative perspective. Given $d,n_1,\ldots,n_d \in \N$, let $K^{(d)}_{n_1,\ldots,n_d}$ denote the complete $d$-uniform $d$-partite hypergraph on with parts $V_1=[n_1], \ldots, V_d=[n_d]$. Edges of $K^{(d)}_{n_1,\ldots,n_d}$ correspond in the obvious way to vertices of the $d$-dimensional grid $[n_1]\times \ldots \times [n_d]$. Subhypergraphs of $K^{(d)}_{n_1,\ldots,n_d}$ of the form $K^{(d)}_{m_1,\ldots,m_d}$ then correspond to subgrids of $[n_1]\times \ldots \times [n_d]$ of size $m_1\times \ldots \times m_d$. Using this correspondence, the following \Cref{lem:grid-2d,lem:grid} follow, for example, from \cite[Theorem 2]{Fur96} and \cite[Theorem 4]{CFS09}, respectively. 

\begin{lem}\label{lem:grid-2d}
Given $k,n,t \in \N$ in any vertex $k$-colouring of an $tk \binom{nk}{n} \times kn$ grid there is a monochromatic subgrid of size $t \times n$.
\end{lem}

In the higher dimensional case we will not need an asymmetric variant.
\begin{lem}\label{lem:grid}
Given integers $d,k \ge 2$ there exists a positive constant $C=C(d,k)$ such that for any positive integers $n$ and $N$ with $N \ge 2^{Cn^{d-1}}$, in any $k$-colouring of the $d$-dimensional $N \times \ldots \times N$ grid there is a monochromatic subgrid of size $n \times \ldots \times n$.
\end{lem}


\subsection{Proofs of the main results on monotonicity}
We begin with the 2-dimensional case. We will actually prove 
a stronger version of \Cref{thm:monotone} \ref{itm:2d} as we will need it for the 3-dimensional case.
\begin{thm}
\label{thm:monotone-asym-2d}
For every $n,t \in \N$, any $4n^2 \times (2t)^{2^{2n}}$ array contains an $n \times t$ monotone subarray. 
\end{thm}
\begin{proof}
Let $f$ be an array indexed by $[N]\times [M]$, where $N=4n^2$ and $M=(2t)^{2^{2n}}$.
By Erd\H{o}s-Szekeres Theorem we know that in each column of $f$ there is a monotone subsequence of length $2n$. The entries of this subsequence can appear in $\binom{4n^2}{2n}$ different positions so there must be a set $R\subseteq [N]$ of $2n$ positions for which at least $M/\binom{4n^2}{2n}$ columns are monotone when restricted to (rows) $R$. We take $C\subseteq [N]$ to be the subset of these columns for which the restriction is increasing, we may w.l.o.g. assume that $C$ consists of at least half of these columns. We obtain a subarray $f'=f|_{R \times C}$ which is increasing in each column and has size $2n \times M'$ where $M' \ge \frac{M}{2\binom{4n^2}{2n}}\ge t^{2^{2n}}$.

By applying Erd\H{o}s-Szekeres Theorem to the sequence given by the first row of $f'$, we can find a subset $C_1 \subseteq C$ of size $|C_1| \ge \sqrt{|C|}$ such that the first row of $f'|_{R \times C_1}$ is monotone. Repeating this argument at step $i$ we find a subset $C_i \subseteq C_{i-1}$ of size $|C_i| \ge \sqrt{|C_{i-1}|} \ge |C|^{1/2^i}$ such that the first $i$ rows of $f'|_{R \times C_i}$ are monotone. Continuing this process until $i=2n$ we obtain an $2n \times t$ array with each row being either increasing or decreasing. By taking the ones of the type which appears more often we obtain a monotone $n \times t$ subarray as claimed.
\end{proof}

One can easily generalise the above proof to give a bound of the form $M_d(n)\le \towr_{d+1}(O(n^{d-1}))$ for any $d\ge 2$, which would already be a substantial improvement over the Ackermann bound due to Fishburn and Graham \cite{FG93}. However, to prove the desired significantly better bound $M_d(n)\le \towr_4(O(n^{d-1}))$ for $d\ge 4$, we need to consider an intermediate problem which we find interesting in its own right.

\begin{defn*}[Inconsistently monotone array]
 An array $f\colon A_1\times \ldots \times A_d \rightarrow \mathbb{R}$ is inconsistently monotone if for each $i\in [d]$, $f(a_1,\ldots,a_{i-1},x,a_{i+1},\ldots,a_d)$ is monotone in $x$ for all choices of $a_1,\ldots, a_{i-1},a_{i+1},\ldots, a_d.$ 
\end{defn*}

Main difference compared with the definition of a monotone array is that we do not require all the lines along a fixed dimension to be all increasing or all decreasing but allow some to be increasing and some to be decreasing. For positive integers $d$ and $n$, let $M'_d(n)$ denote the minimum $N$ such that for any $d$-dimensional array of size $N\times \ldots \times N$, one can find a $d$-dimensional inconsistently monotone subarray of size $n\times \ldots \times n$. We have $M'_1(n)=(n-1)^2+1$ according to Erd\H{o}s-Szekeres Theorem. When $d \ge 2$ we obtain the following version of \Cref{thm:monotone} which gives stronger bounds but only guarantees us an inconsistently monotone array, a weaker notion compared to actual monotone arrays.

\begin{thm}
\label{thm:non-consistent-monotone}
For every $d \ge 2$, we have $M'_d(n) \le 
2^{2^{(1+o(1))n^{d-1}}}$.
\end{thm}

\begin{proof}
We will prove the following recursive bound
\begin{equation}\label{eq:monotone-recursive}
M'_d(n) \le \binom{M'_{d-1}(n)}{n}^{d-1} n^{2^{n^{d-1}}}. 
\
\end{equation}

Let $m=M'_{d-1}(n)$ and $N=\binom{m}{n}^{d-1} n^{2^{n^{d-1}}}$. To prove \eqref{eq:monotone-recursive}, let $f$ be an array indexed by $[m]^{d-1}\times [N]$. For each ``height'' $h \in [N]$, consider the restriction of $f$ to $[m]^{d-1}\times \{h\}$. As $m=M'_{d-1}(n)$, there exist an $n\times \ldots \times n$ subgrid $\vect{S}_h$ of $[m]^{d-1}$ such that $f$ is inconsistently monotone on $\vect{S}_h\times \{h\}$. Given $h \in [N]$, there are at most $\binom{m}{n}^{d-1}$ possibilities for the location of $\vect{S}_h$. Hence, by the pigeonhole principle, we can find an $n\times \ldots \times n$ subgrid $\vect{S}$ of $[m]^{d-1}$ and a subset $H \subset [N]$ of size
\[ 
|H| \ge \frac{N}{\binom{m}{n}^{d-1}}= n^{2^{n^{d-1}}}
\]
such that $f$ is inconsistently monotone on $\vect{S}\times \{h\}$ for every $h \in H$. Let us denote the elements of $\vect{S}$ by $s_1,\ldots, s_{n^{d-1}}$. By Erd\H{o}s-Szekeres Theorem, we can construct a nested sequence $H_0:=H \supseteq H_1 \supseteq \ldots \supseteq H_{n^{d-1}}$ such that $|H_i| \ge \sqrt{|H_{i-1}|}$ for every $i\ge 1$, and that $\{s_j\} \times H_i$ is monotone for $j=1,\ldots,i$. In particular, we have that 
$|H_{n^{d-1}}| \ge |H|^{1/2^{n^{d-1}}} \ge n$, and that the restriction of $f$ to $\vect{S}\times H_{n^{d-1}}$ is inconsistently monotone. This completes the proof of \eqref{eq:monotone-recursive}.  

What remains to be shown is that \eqref{eq:monotone-recursive} implies the desired bound $M'_d(n) \le \towr_3((1+o(1))n^{d-1})$. We proceed by induction on $d$, noting that the case $d=2$ follows from \eqref{eq:monotone-recursive} and the fact that $M'_1(n)=(n-1)^2+1$. For the induction step, assuming $d\ge 3$. Using \eqref{eq:monotone-recursive} and the induction hypothesis we find
\begin{align*}
M'_d(n) &\le M'_{d-1}(n)^{(d-1)n} \cdot n^{2^{n^{d-1}}}\\
& \le \towr_3(O(n^{d-2})) \cdot \towr_3((1+o(1))n^{d-1})\\
&=\towr_3((1+o(1))n^{d-1}),
\end{align*}
finishing the proof.
\end{proof}

The following definition is going to help us find a monotone array inside an inconsistently monotone array.

\begin{defn}[Monotonicity pattern]
Let $f\colon \vect{A} \rightarrow \mathbb{R}$ be an inconsistently monotone $d$-dimensional array. Let $\vect{a}=(a_1,\ldots,a_d) \in \vect{A}$. For each $i\in [d]$, let $s_i \in \{-1,1\}$ be such that $s_if(a_1,\ldots,a_{i-1},x,a_{i+1},\ldots,a_d)$ is increasing in $x$. The vector $\vect{s}=(s_1,\ldots,s_d)$ is called the \emph{monotonicity pattern} of $f$ at point $\vect{a}$. 
\end{defn}

Notice that if $f$ is a monotone array then $f$ has the same monotonicity pattern at all points, in which case we just call it the monotonicity pattern of $f$. We now use Theorem \ref{thm:non-consistent-monotone} to prove Part \ref{itm:4+d} of \Cref{thm:monotone}.
\begin{thm}\label{thm:monotone-anyd}
For every $d \ge 4$, we have $M_d(n) \le \towr_4\left(O(n^{d-1})\right)$.
\end{thm}
\begin{proof}
Let $N=\towr_4\left(O(n^{d-1})\right)$, and let $C=C(d,2^d)$ be the positive constant given by \Cref{lem:grid}. It follows from \Cref{thm:non-consistent-monotone} that in any $d$-dimensional array of size $N\times \ldots \times N$, one can find an inconsistently monotone subarray $f$ indexed by $\vect{A}=A_1\times \ldots \times A_d$ such that $|A_1|=\ldots = |A_d|=2^{C n^{d-1}}.$ 

Let us colour every point in $\vect{A}$ with the monotonicity pattern of $f$ at this point. 
This gives us a vertex-colouring of $\vect{A}$ with $2^d$ colours given by $\{-1,1\}^d$.
By \Cref{lem:grid} and the choice of $C$, there exists a monochromatic $n \times \ldots \times n$ subgrid $\vect{B}$ of $\vect{A}$ with colour $(s_1,\ldots, s_d)$.
From the definition of the monotonicity pattern $(s_1,\ldots,s_d)$, we can see that $f\rvert_{\vect{B}}$ is monotone.
\end{proof}

In order to prove $M_3(n) \le 2^{2^{(2+o(1))n^2}}$, we devise a different argument, not going through the intermediate problem of bounding $M'_d(n)$. 

\begin{thm}
\label{thm:monotone-3d}
We have $M_3(n) \le 
2^{2^{(2+o(1))n^2}}$.
\end{thm}
\begin{proof}
Let $X_1=[16n^2], X_2=[2^{2^{6n}}]$ and $X_3=[2^{2^{(2+o(1))n^2}}]$. To prove \Cref{thm:monotone} \ref{itm:3d} it suffices to show that any 3-dimensional array $f$ indexed by $X_1\times X_2 \times X_3$ contains an $n \times n \times n$ monotone subarray.

For each ``height'' $h\in X_3$, let $\vect{C}_h=X_1 \times X_2 \times \{h\}$. As $2^{2^{6n}} \ge (2n2^{2n})^{2^{4n}}$, \Cref{thm:monotone-asym-2d} implies that each $\vect{C}_h$ contains a $2n \times n 2^{2n} \times 1$ monotone subarray.
There are $\binom{16n^2}{2n}\binom{2^{2^{6n}}}{n2^{2n}}$ different possibilities for a $2n \times n 2^{2n}\times 1$ monotone subarray, and the monotonicity pattern of each such subarray is a vector $\vect{s}\in \{-1,1\}^2$. Since $4\binom{16n^2}{2n}\binom{2^{2^{6n}}}{n2^{2n}}=2^{2^{o(n^2)}}$, by pigeonhole principle, we can find a vector $\vect{s} \in \{-1,1\}^2$ and three subsets $S_1 \subseteq X_1, S_2\subseteq X_2,S_3\subseteq X_3$ with $|S_1|=2n,$ $|S_2|=n2^{2n}$ and $|S_3|=2^{2^{(2+o(1))n^2}}$ such that for any $h\in S_3$ the array $f|_{S_1\times S_2 \times \{h\}}$ is monotone with pattern $\vect{s}$. Our remaining goal is to find an $n \times n$ subgrid of $S_1 \times S_2$ such that for any pair $(a_1,a_2)$ of this subgrid, $f(a_1,a_2,\cdot)$ is always increasing or always decreasing on some fixed subset of size $n$ of $S_3$.

For each $h\in S_3$, let $\vect{L}_h=S_1\times S_2 \times \{h\}$. 
We can think of $\vect{L}_h$'s as ``layers'' stacked one on top of each other. Given two layers $\vect{L}_h$ and $\vect{L}_{h'}$ with $h<h'$, we colour an element $\vect{v} \in S_1 \times S_2$ in red if $f(\vect{v},h)>f(\vect{v},h')$, and blue otherwise. This way we obtain a colouring of $S_1\times S_2$ with two colours, so by \Cref{lem:grid-2d} we can find a monochromatic subgrid $\vect{B}_{hh'}$ of size $n\times n$. We now consider the following edge-colouring of the complete graph on the vertex set $S_3$ using $k$ colours. We colour the edge between $h$ and $h'$ by 
a pair made of $\vect{B}_{hh'}$ and its monochromatic colour. Since there are at most $\binom{2n}{n}\binom{n2^{2n}}{n}$ possibilities for $\vect{B}_{hh'}$, we must have $k \le 2\binom{2n}{n}\binom{n2^{2n}}{n}=2^{(2+o(1))n^2}$, giving $k^{kn} \le 2^{2^{(2+o(1))n^2}}=|S_3|$. From this and a result of Erd\H{o}s and Rado \cite[Theorem 1]{erdos52} on the multicolor Ramsey numbers which states that any $k$-edge colouring of the complete graph on ${k^{kn}}$ many vertices contains a monochromatic $K_n$, we deduce that our colouring contains a monochromatic $K_n$. Let $H\subseteq S_3$ correspond to the vertices of this $K_n$ and let its colour correspond to an $n \times n$ subgrid $\vect{B}$ and say w.l.o.g. blue. This means that $f(a_1,a_2, \cdot): H \to \bR$ is increasing for all $(a_1,a_2) \in \vect{B}$. So by our construction of $S_1,S_2$ we have that $f$ when restricted to $\vect{B} \times H$ is a monotone array of size $n \times n \times n$.
\end{proof}

\textbf{Remark.}
In the above proof we used the usual Ramsey theorem on our colouring of the complete graph on $S_3$. However, our colouring is not arbitrary and in fact one can instead use the ordered Ramsey number of a path (see \cite{ordered-ramsey}). The third alternative is to only colour an edge according to $\vect{B}_{hh'}$,
and record whether the values increase or decrease between $\vect{B}_{hh'} \times \{h\}$ and $\vect{B}_{hh'} \times \{h'\}$ by a directed edge. This gives us a colouring of a tournament in which we want to find a monochromatic directed path (see \cite{Chvatal72,G-L}). 
Both approaches give slightly better bounds than the one in \Cref{thm:monotone} (ii), but unfortunately still give bounds of the form $M_3(n) \le \towr_3(O(n^2))$.

\section{Lexicographic arrays}
\label{sec:lexicographic}
In this section we show our bounds on $\ML_d(n)$, in particular we prove \Cref{thm:mon-to-lex}.

\subsection{Preliminaries}
A monotone array $f$ is said to be \textit{increasing} if restriction of $f$ to any axis parallel line is an increasing sequence (i.e.\ case {\rm (i)} of the definition of monotonicity always occurs). More formally:

\begin{defn}
A $d$-dimensional array $f$ is \textit{increasing} if $f(\vect{x}) \le f(\vect{y})$ whenever $x_i \le y_i$ for all $i\in [d]$.
\end{defn}

The following definition generalises the notion of a lexicographic array to allow for a custom priority order of coordinates.

\begin{defn}
Given a $d$-dimensional array $f$ and a permutation $\sigma \in \mathfrak{S}_d$, we say $f$ is \textit{(lexicographic) of type $\sigma$} if $f(\vect{x})<f(\vect{y}) \Leftrightarrow (x_{\sigma(1)},\ldots,x_{\sigma(d)})<_{\hbox{lex}} (y_{\sigma(1)},\ldots,y_{\sigma(d)})$ for all possible $\vect{x}$ and $\vect{y}$.
\end{defn}

Recall that an array is said to be lex-monotone if it is possible to permute the coordinates and reflect the array along some dimensions to obtain a lexicographic array. The above definition allows us to separate these two actions. In particular, an alternative definition is that an array is lex-monotone if one can reflect the array along some dimensions to obtain a lexicographic array of some type.

Notice that any subarray of a monotone array is itself monotone and moreover has the same monotonicity pattern. This means that when looking for a lex-monotone subarray  within a monotone array we can only ever find one with the same monotonicity pattern. In other words we may w.l.o.g.\ assume that the starting array is increasing. The following immediate lemma makes this statement formal.

\begin{lem}\label{lem:mon-to-inc}
For every $d,n \in \N$, $\ML_d(n)$ equals the minimum $N$ such that any increasing $d$-dimensional array of size $N\times \ldots \times N$ contains an $n\times\ldots \times n$ subarray of type $\sigma$ for some $\sigma \in \mathfrak{S}_d$.
\end{lem}

\subsection{2-dimensional case}
Notice first that in 2 dimensions there are only two possible types of a (lexicographic) array, namely (1,2) and (2,1). See \Cref{fig:types} for an illustration of both together with an example of the arrow notation which we found useful when thinking about the problem.
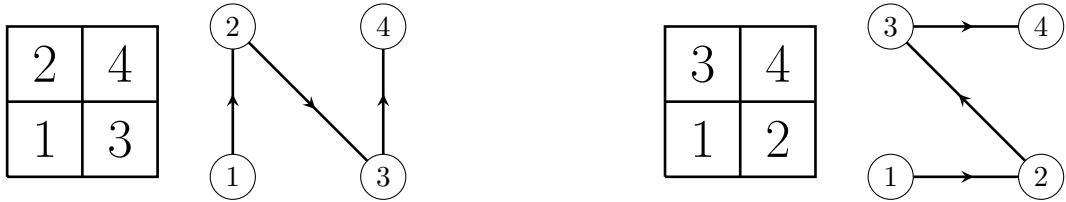
\begin{figure}[h]
\begin{minipage}{0.49\textwidth}
\begin{center}
    \begin{tikzpicture}[scale=2]
   \draw[line width = 1pt] (-1.5,0) -> (-0.5,0) -> (-0.5,1) -> (-1.5,1) -> (-1.5,0);
    \draw[line width = 1pt] (-1,0) -> (-1,1);
    \draw[line width = 1pt] (-1.5,0.5) -> (-0.5,0.5);
    \node[] at (-1.25,0.25) {\huge{\textbf{$1$}}};
    \node[] at (-0.75,0.25) {\huge{\textbf{$3$}}};
    \node[] at (-1.25,0.75) {\huge{\textbf{$2$}}};
    \node[] at (-0.75,0.75) {\huge{\textbf{$4$}}};

    \draw[diredge, line width = 1pt] (1,0) -- (1,1);
    \draw[diredge, line width = 1pt] (0,0) -- (0,1);
    \draw[diredge, line width = 1pt] (0,1) -- (1,0);
    \draw[] (0,0) \vx;
    \node[] at (0,0) {{\textbf{$1$}}};
    
    \draw[] (0,1) \vx;
    \node[] at (0,1) {{\textbf{$2$}}};
    
    \draw[] (1,0) \vx;
    \node[] at (1,0) {{\textbf{$3$}}};
    
    \draw[] (1,1) \vx;
    \node[] at (1,1) {{\textbf{$4$}}};

    \end{tikzpicture}
    
\end{center}
\end{minipage} 
\begin{minipage}{0.49\textwidth}
\begin{center}
    \begin{tikzpicture}[scale=2]
    
    \draw[line width = 1pt] (-1.5,0) -> (-0.5,0) -> (-0.5,1) -> (-1.5,1) -> (-1.5,0);
    \draw[line width = 1pt] (-1,0) -> (-1,1);
    \draw[line width = 1pt] (-1.5,0.5) -> (-0.5,0.5);
    \node[] at (-1.25,0.25) {\huge{\textbf{$1$}}};
    \node[] at (-0.75,0.25) {\huge{\textbf{$2$}}};
    \node[] at (-1.25,0.75) {\huge{\textbf{$3$}}};
    \node[] at (-0.75,0.75) {\huge{\textbf{$4$}}};
    
    \draw[diredge, line width = 1pt] (0,0) -> (1,0); 
    \draw[diredge, line width = 1pt] (0,1) -> (1,1);
    \draw[diredge, line width = 1pt] (1,0) -> (0,1);
    \draw[] (0,0) \vx;
    \node[] at (0,0) {{\textbf{$1$}}};
    
    \draw[] (0,1) \vx;
    \node[] at (0,1) {{\textbf{$3$}}};
    
    \draw[] (1,0) \vx;
    \node[] at (1,0) {{\textbf{$2$}}};
    
    \draw[] (1,1) \vx;
    \node[] at (1,1) {{\textbf{$4$}}};
    \end{tikzpicture}
    
\end{center}
\end{minipage} 

\caption{Lexicographic arrays of type (1,2) and (2,1), arrows point towards larger points.}
    \label{fig:types}
\end{figure}

We begin with a proof of the upper bound $\ML_2(n) \le 2n^2-5n+4$ as it sheds some light to where our lower bound construction is coming from. 

\begin{thm}[Fishburn and Graham \cite{FG93}]
\label{thm:F-G}
For $n \in \N$ we have $\ML_2(n)\le 2n^2-5n+4$.
\end{thm}
\begin{proof}
Let $f$ be an increasing array indexed by $[N] \times [N]$, where $N=(n-1)(2n-3)+1$. For $i \in [2n-2]$, let $a_i=(n-1)(i-1)+1$. Define a red-blue colouring of the grid $\{a_1,\ldots,a_{2n-3}\}\times \{a_1,\ldots,a_{2n-3}\}$ as follows. For every $i,j \in [2n-3]$, we colour $(a_i,a_j)$ red if $f(a_{i+1},a_j) < f(a_i,a_{j+1})$, and blue otherwise. As $(n-2)(2n-3)+(n-2)(2n-3)<(2n-3)^2$, there exists a row with at least $n-1$ red points or a column with at least $n-1$ blue points. By symmetry, we can assume $(a_i,a_{j_1}),\ldots, (a_i,a_{j_{n-1}})$ are $n-1$ red points in a row $a_i$ with $j_1<j_2< \ldots<j_{n-1}$. One can check that the $n\times n$ subarray of $f$ indexed by $[a_i,a_{i+1}] \times \{a_{j_1}, \ldots, a_{j_{n-1}},a_{j_{n-1}+1}\}$ is of type $(1,2)$. Hence $F_2(n) \le N=2n^2-5n+4$, as required.
\end{proof}

The remainder of this subsection is devoted to the proof of the lower bound $F_2(d) \ge 2n^2-5n+4$. We will make ample use of the immediate observation that any subarray of a lexicographic array of type $\sigma$ which has size at least $2$ in each dimension must also be of type $\sigma$. We first construct a ``building block'' for our actual construction showing $F_2(d) \ge 2n^2-5n+4$.

\begin{lem}\label{lem:block1}
For $n\ge 3$, there exists an increasing array $g$ of size $(n-1)(n-2)\times (n-1)^2$ such that
\begin{enumerate}
    \item[\rm (G1)] $g$ does not contain an $(n-1) \times 2$ subarray of type $(1,2)$,
    \item[\rm (G2)] $g$ does not contain an $n \times 2$ subarray of type $(2,1)$.
\end{enumerate}
\end{lem}

\begin{proof}
For $1 \le i \le n-2$, let $\vect{C}_i=[(n-1)(i-1)+1, (n-1)i] \times [(n-1)^2]$. We choose an array $g$ (see \Cref{fig:g} for an illustration), indexed by $[(n-1)(n-2)]\times [(n-1)^2]$, such that 
\begin{itemize}
    \item $g\rvert_{\vect{C_1}}< \ldots < g \rvert_{\vect{C}_{n-2}}$,
    \item For each $i \in [n-2]$, $g\rvert_{\vect{C}_i}$ is of type $(2,1)$.
\end{itemize}

\begin{figure}[h]
    \centering
    \begin{tikzpicture}[scale=0.55]

\def \x {2.5}
\def \dx {0.25}
\def \y {4*\x+3*\dx}
\def \ys {3*\x+3*\dx}
\def \epsi {0.1}

\defPt{0}{0}{v1}
\defPt{\x}{0}{v2}
\defPt{\x+\dx}{0}{v3}
\defPt{2*\x+\dx}{0}{v4}
\defPt{3*\x+2*\dx}{0}{v5}
\defPt{4*\x+2*\dx}{0}{v6}
\defPt{0}{\y}{u1}
\defPt{\x}{\y}{u2}
\defPt{\x+\dx}{\y}{u3}
\defPt{2*\x+\dx}{\y}{u4}
\defPt{3*\x+2*\dx}{\y}{u5}
\defPt{4*\x+2*\dx}{\y}{u6}

\draw [thick, black, decorate,decoration={brace,amplitude=6pt,  mirror}] ($(v1)+(0,-\epsi)$) -- ($(v2)+(0,-\epsi)$) node[black,midway,yshift=-0.5cm] { $n-1$};

\draw [thick, black, decorate,decoration={brace,amplitude=6pt,  mirror}] ($(v3)+(0,-\epsi)$) -- ($(v4)+(0,-\epsi)$) node[black,midway,yshift=-0.5cm] { $n-1$};

\draw [thick, black, decorate,decoration={brace,amplitude=6pt,  mirror}] ($(v5)+(0,-\epsi)$) -- ($(v6)+(0,-\epsi)$) node[black,midway,yshift=-0.5cm] { $n-1$};

\draw [thick, black, decorate,decoration={brace,amplitude=10pt}] ($(v1)+(-\epsi,0)$) -- ($(u1)+(-\epsi,0)$) node[black,midway,xshift=-1.2cm] { $(n-1)^2$};

\draw[diredge, line width=1 pt] (v1) -- (v2);
\draw[diredge, line width=1 pt] (v3) -- (v4);
\draw[diredge, line width=1 pt] (v5) -- (v6);

\draw[diredge, line width=1 pt] (u1) -- (u2);
\draw[diredge, line width=1 pt] (u3) -- (u4);
\draw[diredge, line width=1 pt] (u5) -- (u6);

\draw[diredge, line width=1 pt] (v2)-- (u1);
\draw[diredge, line width=1 pt] (v4)-- (u3);
\draw[diredge, line width=1 pt] (v6)-- (u5);

    \foreach \j in {1,...,6}
    {
        \draw[line width=1 pt] (v\j) -- (u\j);
    }
    
    \node[] at ($0.395*(u1)+0.12*(v1)+0.395*(u2)+0.12*(v2)$) {{\textbf{$1$}}};
    
    \node[] at ($0.395*(u3)+0.12*(v3)+0.395*(u4)+0.12*(v4)$) {{\textbf{$2$}}};
    
    \node[] at ($0.5*(u3)+0.5*(v6)$) {\Huge{\textbf{$\cdots$}}};
    
    \node[] at ($0.395*(u5)+0.12*(v5)+0.395*(u6)+0.12*(v6)$) {{\textbf{$n-2$}}};    
\end{tikzpicture}
    \caption{An illustration of the array $g$. Directed arrows point towards a position with a larger value of $g$. Numbers denote relative order of subarrays.}
    \label{fig:g}
\end{figure}
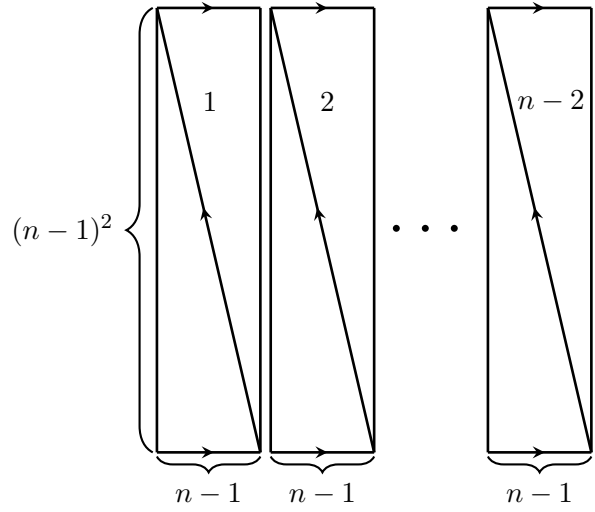

For (G1), if such a subarray exists, then at least one of $\vect{C}_i$'s would need to intersect this subarray in some $2 \times 2$ subarray. By the second property of $g$, the $2\times 2$ subarray is of type $(2,1)$, a contradiction. 

For (G2), if such a subarray exists it would intersect at least two distinct $\vect{C}_i$'s, and so it would contain a $2\times 2$ subarray of type $(1,2)$, a contradiction.
\end{proof} 

Another building block of our construction is the following.
\begin{lem}\label{lem:block2}
For $n\ge 3$, there is an increasing array $h$ of size $(n-1)^2\times (n-1)(n-2)$ such that
\begin{enumerate}
    \item[\rm (H1)] $h$ does not contain a $2 \times n$ subarray of type $(1,2)$,
    \item[\rm (H2)] $h$ does not contain a $2 \times (n-1)$ subarray of type $(2,1)$.
    \end{enumerate}
\end{lem}

\begin{proof}
Let $\vect{R}_i= [(n-1)^2] \times [(n-1)i-n+1, (n-1)i]$ for $i \in [n-2]$. Let us define $h$ to be an array indexed by $[(n-1)^2]\times [(n-1)(n-2)]$ so that $h \rvert_{\vect{R}_i}<h \rvert_{\vect{R}_j}$ whenever $i<j$ and so that $h\rvert_{\vect{R}_i}$ is of type $(1,2)$ (see \Cref{fig:h} for an illustration).
This array satisfies the properties (H1) and (H2) by the same argument as in \Cref{lem:block1}.
\end{proof}

\begin{figure}[h]
    \centering
    \begin{tikzpicture}[scale=0.55]

\def \x {2.5}
\def \dx {0.25}
\def \y {4*\x+3*\dx}
\def \ys {3*\x+3*\dx}
\def \epsi {0.1}

\defPt{3*\x+3*\dx}{0}{x1}
\defPt{3*\x+3*\dx+\y}{0}{x2}
\defPt{3*\x+3*\dx}{\x}{y1}
\defPt{3*\x+3*\dx+\y}{\x}{y2}
\defPt{3*\x+3*\dx}{\x+\dx}{x3}
\defPt{3*\x+3*\dx+\y}{\x+\dx}{x4}
\defPt{3*\x+3*\dx}{2*\x+\dx}{y3}
\defPt{3*\x+3*\dx+\y}{2*\x+\dx}{y4}
\defPt{3*\x+3*\dx}{3*\x+2*\dx}{x5}
\defPt{3*\x+3*\dx+\y}{3*\x+2*\dx}{x6}
\defPt{3*\x+3*\dx}{4*\x+2*\dx}{y5}
\defPt{3*\x+3*\dx+\y}{4*\x+2*\dx}{y6}

\draw [thick, black, decorate,decoration={brace,amplitude=6pt}] ($(x5)+(-\epsi,0)$) -- ($(y5)+(-\epsi,0)$) node[black,midway,xshift=-0.7 cm] { $n-1$};

\draw [thick, black, decorate,decoration={brace,amplitude=6pt}] ($(x3)+(-\epsi,0)$) -- ($(y3)+(-\epsi,0)$) node[black,midway,xshift=-0.7 cm] { $n-1$};

\draw [thick, black, decorate,decoration={brace,amplitude=6pt}] ($(x1)+(-\epsi,0)$) -- ($(y1)+(-\epsi,0)$) node[black,midway,xshift=-0.7 cm] { $n-1$};

\draw [thick, black, decorate,decoration={brace,amplitude=10pt,mirror}] ($(x1)+(0,-\epsi)$) -- ($(x2)+(0,-\epsi)$) node[black,midway,yshift=-0.6 cm] { $(n-1)^2$};

\draw[line width=1 pt] (x1) -- (x2);
\draw[line width=1 pt] (x3) -- (x4);
\draw[line width=1 pt] (x5) -- (x6);

\draw[line width=1 pt] (y1) -- (y2);
\draw[line width=1 pt] (y3) -- (y4);
\draw[line width=1 pt] (y5) -- (y6);

\draw[diredge, line width=1 pt] (y1) -- (x2);
\draw[diredge, line width=1 pt]  (y3) -- (x4);
\draw[diredge, line width=1 pt]  (y5) --(x6);

    \foreach \j in {1,...,6}
    {
         \draw[diredge, line width=1 pt] (x\j) -- (y\j);
    }
    
    \node[] at ($0.395*(x2)+0.12*(x1)+0.395*(y2)+0.12*(y1)$) {\Large{\textbf{$1$}}};
    
    \node[] at ($0.395*(x4)+0.12*(x3)+0.395*(y4)+0.12*(y3)$) {\Large{\textbf{$2$}}};
    
    \node[] at ($0.5*(y3)+0.5*(x6)+(0,0.35)$) {\Huge{\textbf{$\vdots$}}};
    
    \node[] at ($0.395*(x6)+0.12*(x5)+0.395*(y6)+0.12*(y5)$) {\Large{\textbf{$n-2$}}};    
\end{tikzpicture}
    \caption{An illustration of the array $h$. Directed arrows point towards a position with a larger value of $h$. Numbers denote relative order of subarrays.}
    \label{fig:h}
\end{figure}
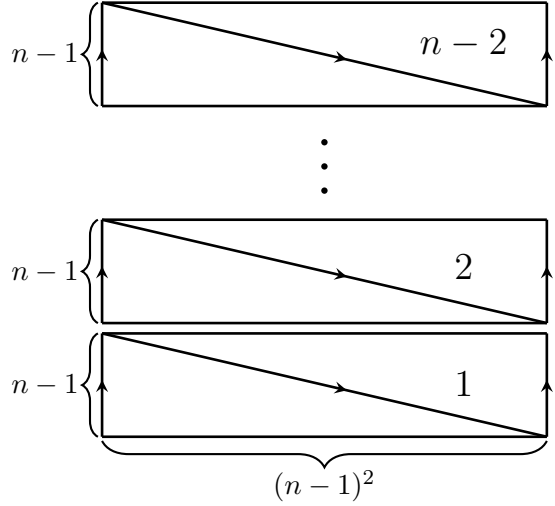

We are now in a position to prove \Cref{thm:mon-to-lex} (i).

\begin{thm}
For $n \in \N$ we have $\ML_2(n) \ge 2n^2-5n+4$.
\end{thm}

\begin{proof}
It is immediate that the statement holds for $n=1,2$. We henceforth assume that $n \ge 3$.

Let $N=2n^2-5n+3=(n-1)^2+(n-1)(n-2)$. To prove the statement, it suffices to construct an increasing array $f:[N]^2\rightarrow \bR$ which does not contain an $n\times n$ subgrid of type (1,2) or (2,1).

We first split $[N]^2$ into five subgrids $\vect{A}_1,\ldots,\vect{A}_5$ (see \Cref{fig:f}) such that both $\vect{A}_1$ and $\vect{A}_5$ have size $(n-1)(n-2)\times (n-1)^2$, both $\vect{A}_2$ and $\vect{A}_4$ have size $(n-1)^2\times (n-1)(n-2)$, while $\vect{A}_3$ has size $(n-1)\times (n-1)$. 
Let $g$ and $h$ be arrays given by \Cref{lem:block1} and \Cref{lem:block2}, respectively. The array $f$ is chosen so that $f\rvert_{\vect{A}_1}<f\rvert_{\vect{A}_2}<\ldots < f\rvert_{\vect{A}_5}$, $f\rvert_{\vect{A}_1}$ and $f\rvert_{\vect{A}_5}$ are copies of $g$, $f\rvert_{\vect{A}_2}$ and $f\rvert_{\vect{A}_4}$ are copies of $h$, and $f\rvert_{\vect{A}_3}$ is an arbitrary increasing array. Since $f\rvert_{\vect{A}_1}<f\rvert_{\vect{A}_2}<\ldots < f\rvert_{\vect{A}_5}$ and $f\rvert_{\vect{A}_i}$ is increasing for every $1\le i \le 5$, $f$ is increasing as well. It remains to show that $f$ does not contain an $n \times n$ subarray of type (1,2) or (2,1).

As $f\rvert_{\vect{A}_1}<f\rvert_{\vect{A}_2}<\ldots < f\rvert_{\vect{A}_5}$ and $f\rvert_{\vect{A}_i}$ is increasing for every $1\le i \le 5$, we also find that
\begin{enumerate}
	\item[(P1)] Any $2\times 2$ subarray with two vertices in $\vect{A}_1$ and two vertices to the right of $\vect{A}_1$ is of type (1,2),
	\item[(P2)] Any $2\times 2$ subarray with two vertices in $\vect{A}_5$ and two vertices to the left of $\vect{A}_5$ is of type (1,2),
	\item[(P3)] Any $2\times 2$ subarray with two vertices in $\vect{A}_2$ and two vertices above $\vect{A}_2$ is of type (2,1),
	\item[(P4)] Any $2\times 2$ subarray with two vertices in $\vect{A}_4$ and two vertices below $\vect{A}_4$ is of type (2,1).
\end{enumerate}
We will show that $f$ has the desired property using properties (P1)--(P4) together with conditions (G1), (G2), (H1), (H2) from Lemmas \ref{lem:block1} and \ref{lem:block2}.

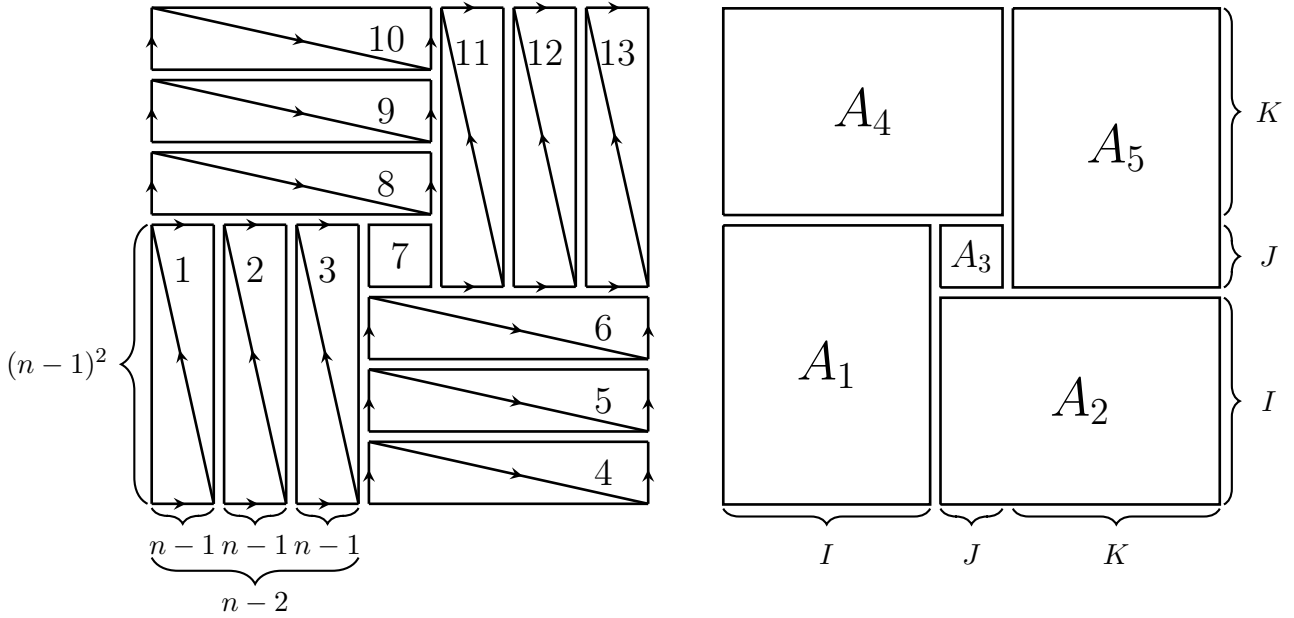
\begin{figure}[h]\centering
\begin{minipage}[b]{0.48\textwidth}\centering
\begin{tikzpicture}[scale=0.55]

\def \x {1.5}
\def \dx {0.25}
\def \y {4*\x+3*\dx}
\def \ys {3*\x+3*\dx}
\def \epsi {0.1}

\defPt{0}{0}{v1}
\defPt{\x}{0}{v2}
\defPt{\x+\dx}{0}{v3}
\defPt{2*\x+\dx}{0}{v4}
\defPt{2*\x+2*\dx}{0}{v5}
\defPt{3*\x+2*\dx}{0}{v6}
\defPt{0}{\y}{u1}
\defPt{\x}{\y}{u2}
\defPt{\x+\dx}{\y}{u3}
\defPt{2*\x+\dx}{\y}{u4}
\defPt{2*\x+2*\dx}{\y}{u5}
\defPt{3*\x+2*\dx}{\y}{u6}

\draw [thick, black, decorate,decoration={brace,amplitude=6pt,  mirror}] ($(v1)+(0,-\epsi)$) -- ($(v2)+(0,-\epsi)$) node[black,midway,yshift=-0.5cm] { $n-1$};

\draw [thick, black, decorate,decoration={brace,amplitude=6pt,  mirror}] ($(v3)+(0,-\epsi)$) -- ($(v4)+(0,-\epsi)$) node[black,midway,yshift=-0.5cm] { $n-1$};

\draw [thick, black, decorate,decoration={brace,amplitude=6pt,  mirror}] ($(v5)+(0,-\epsi)$) -- ($(v6)+(0,-\epsi)$) node[black,midway,yshift=-0.5cm] { $n-1$};

\draw [thick, black, decorate,decoration={brace,amplitude=10pt,  mirror}] ($(v1)+(0,-1.3)$) -- ($(v6)+(0,-1.3)$) node[black,midway,yshift=-0.6cm] { $n-2$};

\draw [thick, black, decorate,decoration={brace,amplitude=10pt}] ($(v1)+(-\epsi,0)$) -- ($(u1)+(-\epsi,0)$) node[black,midway,xshift=-1.2cm] { $(n-1)^2$};

\defPt{3*\x+3*\dx}{0}{x1}
\defPt{3*\x+3*\dx+\y}{0}{x2}
\defPt{3*\x+3*\dx}{\x}{y1}
\defPt{3*\x+3*\dx+\y}{\x}{y2}
\defPt{3*\x+3*\dx}{\x+\dx}{x3}
\defPt{3*\x+3*\dx+\y}{\x+\dx}{x4}
\defPt{3*\x+3*\dx}{2*\x+\dx}{y3}
\defPt{3*\x+3*\dx+\y}{2*\x+\dx}{y4}
\defPt{3*\x+3*\dx}{2*\x+2*\dx}{x5}
\defPt{3*\x+3*\dx+\y}{2*\x+2*\dx}{x6}
\defPt{3*\x+3*\dx}{3*\x+2*\dx}{y5}
\defPt{3*\x+3*\dx+\y}{3*\x+2*\dx}{y6}

\defPt{\y+\dx}{\ys}{vv1}
\defPt{\x+\y+\dx}{\ys}{vv2}
\defPt{\x+\dx+\y+\dx}{\ys}{vv3}
\defPt{2*\x+\dx+\y+\dx}{\ys}{vv4}
\defPt{2*\x+2*\dx+\y+\dx}{\ys}{vv5}
\defPt{3*\x+2*\dx+\y+\dx}{\ys}{vv6}
\defPt{\y+\dx}{\y+\ys}{uu1}
\defPt{\x+\y+\dx}{\y+\ys}{uu2}
\defPt{\x+\dx+\y+\dx}{\y+\ys}{uu3}
\defPt{2*\x+\dx+\y+\dx}{\y+\ys}{uu4}
\defPt{2*\x+2*\dx+\y+\dx}{\y+\ys}{uu5}
\defPt{3*\x+2*\dx+\y+\dx}{\y+\ys}{uu6}

\defPt{0}{\y+\dx}{xx1}
\defPt{0\y}{\y+\dx}{xx2}
\defPt{0}{\x+\y+\dx}{yy1}
\defPt{\y}{\x+\y+\dx}{yy2}
\defPt{0}{\x+\dx+\y+\dx}{xx3}
\defPt{\y}{\x+\dx+\y+\dx}{xx4}
\defPt{0}{2*\x+\dx+\y+\dx}{yy3}
\defPt{\y}{2*\x+\dx+\y+\dx}{yy4}
\defPt{0}{2*\x+2*\dx+\y+\dx}{xx5}
\defPt{\y}{2*\x+2*\dx+\y+\dx}{xx6}
\defPt{0}{3*\x+2*\dx+\y+\dx}{yy5}
\defPt{\y}{3*\x+2*\dx+\y+\dx}{yy6}

\draw[diredge, line width=1 pt] (v1) -- (v2);
\draw[diredge, line width=1 pt] (v3) -- (v4);
\draw[diredge, line width=1 pt] (v5) -- (v6);

\draw[diredge, line width=1 pt] (u1) -- (u2);
\draw[diredge, line width=1 pt] (u3) -- (u4);
\draw[diredge, line width=1 pt] (u5) -- (u6);

\draw[diredge, line width=1 pt] (vv1) -- (vv2);
\draw[diredge, line width=1 pt] (vv3) -- (vv4);
\draw[diredge, line width=1 pt] (vv5) -- (vv6);

\draw[diredge, line width=1 pt] (uu1) -- (uu2);
\draw[diredge, line width=1 pt] (uu3) -- (uu4);
\draw[diredge, line width=1 pt] (uu5) -- (uu6);

\draw[line width=1 pt] (x1) -- (x2);
\draw[line width=1 pt] (x3) -- (x4);
\draw[line width=1 pt] (x5) -- (x6);

\draw[line width=1 pt] (y1) -- (y2);
\draw[line width=1 pt] (y3) -- (y4);
\draw[line width=1 pt] (y5) -- (y6);

\draw[line width=1 pt] (xx1) -- (xx2);
\draw[line width=1 pt] (xx3) -- (xx4);
\draw[line width=1 pt] (xx5) -- (xx6);

\draw[line width=1 pt] (yy1) -- (yy2);
\draw[line width=1 pt] (yy3) -- (yy4);
\draw[line width=1 pt] (yy5) -- (yy6);

\draw[diredge, line width=1 pt] (v2)-- (u1);
\draw[diredge, line width=1 pt] (v4)-- (u3);
\draw[diredge, line width=1 pt] (v6)-- (u5);

\draw[diredge, line width=1 pt] (y1) -- (x2);
\draw[diredge, line width=1 pt]  (y3) -- (x4);
\draw[diredge, line width=1 pt]  (y5) --(x6);

\draw[diredge, line width=1 pt] (vv2) -- (uu1);
\draw[diredge, line width=1 pt] (vv4) -- (uu3);
\draw[diredge, line width=1 pt] (vv6) -- (uu5);

\draw[diredge, line width=1 pt] (yy1) -- (xx2);
\draw[diredge, line width=1 pt] (yy3) -- (xx4);
\draw[diredge, line width=1 pt] (yy5) -- (xx6);

\draw[line width=1 pt] (\ys,\ys) -- (\ys+\x,\ys) -- (\ys+\x,\ys+\x) -- (\ys,\ys+\x)-- (\ys,\ys);

    \foreach \j in {1,...,6}
    {
        \draw[line width=1 pt] (v\j) -- (u\j);
        \draw[diredge, line width=1 pt] (x\j) -- (y\j);
        \draw[line width=1 pt] (vv\j) -- (uu\j);
        \draw[diredge, line width=1 pt] (xx\j) -- (yy\j);
    }
    \node[] at ($0.5*(uu6)+0.5*(v1)$) {\Large{\textbf{$7$}}};

    \node[] at ($0.42*(u1)+0.08*(v1)+0.42*(u2)+0.08*(v2)$) {\Large{\textbf{$1$}}};
    
    \node[] at ($0.42*(u3)+0.08*(v3)+0.42*(u4)+0.08*(v4)$) {\Large{\textbf{$2$}}};
    
    \node[] at ($0.42*(u5)+0.08*(v5)+0.42*(u6)+0.08*(v6)$) {\Large{\textbf{$3$}}};
    
    \node[] at ($0.42*(uu1)+0.08*(vv1)+0.42*(uu2)+0.08*(vv2)$) {\Large{\textbf{$11$}}};
    
    \node[] at ($0.42*(uu3)+0.08*(vv3)+0.42*(uu4)+0.08*(vv4)$) {\Large{\textbf{$12$}}};
    
    \node[] at ($0.42*(uu5)+0.08*(vv5)+0.42*(uu6)+0.08*(vv6)$) {\Large{\textbf{$13$}}};
    
    \node[] at ($0.42*(yy2)+0.08*(yy1)+0.42*(xx2)+0.08*(xx1)$) {\Large{\textbf{$8$}}};

    \node[] at ($0.42*(yy4)+0.08*(yy3)+0.42*(xx4)+0.08*(xx3)$) {\Large{\textbf{$9$}}};

    \node[] at ($0.42*(yy6)+0.08*(yy5)+0.42*(xx6)+0.08*(xx5)$) {\Large{\textbf{$10$}}};

    \node[] at ($0.42*(y2)+0.08*(y1)+0.42*(x2)+0.08*(x1)$) {\Large{\textbf{$4$}}};

    \node[] at ($0.42*(y4)+0.08*(y3)+0.42*(x4)+0.08*(x3)$) {\Large{\textbf{$5$}}};

    \node[] at ($0.42*(y6)+0.08*(y5)+0.42*(x6)+0.08*(x5)$) {\Large{\textbf{$6$}}};

\end{tikzpicture}
\end{minipage}\hfill
\begin{minipage}[b]{0.48\textwidth}\centering
\raisebox{0.658 cm}
{
\begin{tikzpicture}[scale=0.55]

\def \x {1.5}
\def \dx {0.25}
\def \y {4*\x+3*\dx}
\def \ys {3*\x+3*\dx}
\def \epsi {0.1}

\defPt{0}{0}{v1}
\defPt{\x}{0}{v2}
\defPt{\x+\dx}{0}{v3}
\defPt{2*\x+\dx}{0}{v4}
\defPt{2*\x+2*\dx}{0}{v5}
\defPt{3*\x+2*\dx}{0}{v6}
\defPt{0}{\y}{u1}
\defPt{\x}{\y}{u2}
\defPt{\x+\dx}{\y}{u3}
\defPt{2*\x+\dx}{\y}{u4}
\defPt{2*\x+2*\dx}{\y}{u5}
\defPt{3*\x+2*\dx}{\y}{u6}

\draw [thick, black, decorate,decoration={brace,amplitude=6pt,  mirror}] ($(v1)+(0,-0.1)$) -- ($(v6)+(0,-0.1)$) node[black,midway,yshift=-0.6cm] { $I$};
\draw [thick, black, decorate,decoration={brace,amplitude=6pt,  mirror}] ($(4*\x+4*\dx,0)+(0,-0.1)$) -- ($(x2)+(0,-0.1)$) node[black,midway,yshift=-0.6cm] { $K$};
\draw [thick, black, decorate,decoration={brace,amplitude=6pt,  mirror}] ($(x1)+(0,-0.1)$) -- ($(4*\x+3*\dx,0)+(0,-0.1)$) node[black,midway,yshift=-0.6cm] { $J$};

\draw [thick, black, decorate,decoration={brace,amplitude=6pt,  mirror}] ($(x2)+(0.1,0)$) -- ($(y6)+(0.1,0)$) node[black,midway,xshift=0.6cm] { $I$};

\draw [thick, black, decorate,decoration={brace,amplitude=6pt,  mirror}] ($(y6)+(0.1,\dx)$) -- ($(y6)+(0.1,\x+\dx)$) node[black,midway,xshift=0.6cm] { $J$};

\draw [thick, black, decorate,decoration={brace,amplitude=6pt,  mirror}] ($(y6)+(0.1,\x+2*\dx)$) -- ($(uu6)+(0.1,0)$) node[black,midway,xshift=0.6cm] { $K$};

\defPt{3*\x+3*\dx}{0}{x1}
\defPt{3*\x+3*\dx+\y}{0}{x2}
\defPt{3*\x+3*\dx}{\x}{y1}
\defPt{3*\x+3*\dx+\y}{\x}{y2}
\defPt{3*\x+3*\dx}{\x+\dx}{x3}
\defPt{3*\x+3*\dx+\y}{\x+\dx}{x4}
\defPt{3*\x+3*\dx}{2*\x+\dx}{y3}
\defPt{3*\x+3*\dx+\y}{2*\x+\dx}{y4}
\defPt{3*\x+3*\dx}{2*\x+2*\dx}{x5}
\defPt{3*\x+3*\dx+\y}{2*\x+2*\dx}{x6}
\defPt{3*\x+3*\dx}{3*\x+2*\dx}{y5}
\defPt{3*\x+3*\dx+\y}{3*\x+2*\dx}{y6}

\defPt{\y+\dx}{\ys}{vv1}
\defPt{\x+\y+\dx}{\ys}{vv2}
\defPt{\x+\dx+\y+\dx}{\ys}{vv3}
\defPt{2*\x+\dx+\y+\dx}{\ys}{vv4}
\defPt{2*\x+2*\dx+\y+\dx}{\ys}{vv5}
\defPt{3*\x+2*\dx+\y+\dx}{\ys}{vv6}
\defPt{\y+\dx}{\y+\ys}{uu1}
\defPt{\x+\y+\dx}{\y+\ys}{uu2}
\defPt{\x+\dx+\y+\dx}{\y+\ys}{uu3}
\defPt{2*\x+\dx+\y+\dx}{\y+\ys}{uu4}
\defPt{2*\x+2*\dx+\y+\dx}{\y+\ys}{uu5}
\defPt{3*\x+2*\dx+\y+\dx}{\y+\ys}{uu6}

\defPt{0}{\y+\dx}{xx1}
\defPt{0\y}{\y+\dx}{xx2}
\defPt{0}{\x+\y+\dx}{yy1}
\defPt{\y}{\x+\y+\dx}{yy2}
\defPt{0}{\x+\dx+\y+\dx}{xx3}
\defPt{\y}{\x+\dx+\y+\dx}{xx4}
\defPt{0}{2*\x+\dx+\y+\dx}{yy3}
\defPt{\y}{2*\x+\dx+\y+\dx}{yy4}
\defPt{0}{2*\x+2*\dx+\y+\dx}{xx5}
\defPt{\y}{2*\x+2*\dx+\y+\dx}{xx6}
\defPt{0}{3*\x+2*\dx+\y+\dx}{yy5}
\defPt{\y}{3*\x+2*\dx+\y+\dx}{yy6}

\draw[line width=1 pt] (u1) -- (u6) -- (v6) -- (v1) -- (u1);
\draw[line width=1 pt] (uu1) -- (uu6) -- (vv6) -- (vv1) -- (uu1);
\draw[line width=1 pt] (x1) -- (x2) -- (y6) -- (y5) -- (x1);
\draw[line width=1 pt] (xx1) -- (xx2) -- (yy6) -- (yy5) -- (xx1);

\draw[line width=1 pt] (\ys,\ys) -- (\ys+\x,\ys) -- (\ys+\x,\ys+\x) -- (\ys,\ys+\x)-- (\ys,\ys);

    \node[] at ($0.5*(u1)+0.5*(v6)$) {\huge{\textbf{$A_1$}}};
    \node[] at ($0.5*(uu1)+0.5*(vv6)$) {\huge{\textbf{$A_5$}}};
    \node[] at ($0.5*(\ys,\ys)+0.5*(\ys+\x,\ys+\x)$) {\Large{\textbf{$A_3$}}};    
    \node[] at ($0.5*(x1)+0.5*(y6)$) {\huge{\textbf{$A_2$}}};
    \node[] at ($0.5*(xx1)+0.5*(yy6)$) {\huge{\textbf{$A_4$}}};

\end{tikzpicture}}
\end{minipage}\hfill

\caption{Directed arrows point towards a larger value of $f$ and indicate whether the given subarray is of type (1,2)  or type (2,1) (as in \Cref{fig:types}). Numbers denote the relative order of subarrays.}
\label{fig:f}
\end{figure}

Suppose towards a contradiction that $[N]^2$ contains an $n\times n$ subgrid $\vect{L}=L_1 \times L_2$ such that $f\rvert_{\vect{L}}$ is of type (1,2) or (2,1). Letting $I=[(n-1)(n-2)]$, $J=[(n-1)(n-2)+1, (n-1)^2]$, $K=[(n-1)^2+1,N]$, we define 
\begin{align*}
a = |L_1\cap I|, \enskip & b=|L_1\cap J|, \enskip c=|L_1\cap K|\\
x = |L_2\cap I|, \enskip & y=|L_2\cap J|, \enskip z=|L_2\cap K|.
\end{align*}
We will obtain various inequalities involving $a,b,c,x,y,z$, and eventually reach a contradiction. Since $\vect{L}$ has size $n\times n$ and $|J|=n-1$  we obtain
\begin{equation}\label{eq:1}
a+b+c=x+y+z=n,\enskip 0 \le a,c,x,z \le n \enskip \text{and} \enskip 0 \le b,y \le n-1.
\end{equation}

We divide our analysis into two cases.

\hspace*{0.3 cm} {\bf Case 1:} $\vect{L}$ is of type $(2,1)$.

We have the following series of observations  
\begin{align}
{\rm (G2)} \Rightarrow & \quad a \le n-1 \text{ or } x+y \le 1,  \label{eq:2}\\
{\rm (G2)} \Rightarrow & \quad c \le n-1 \text{ or } y+z \le 1, \label{eq:3}\\
{\rm (H2)} \Rightarrow & \quad a+b \le 1 \text{ or } z \le n-2,  \label{eq:4}\\
{\rm (H2)} \Rightarrow & \quad b+c \le 1  \text{ or } x \le n-2,  \label{eq:5}\\
{\rm (P1)} \Rightarrow & \quad a = 0 \text{ or } b+c=0 \text{ or } x+y \le 1,   \label{eq:6}\\
{\rm (P2)} \Rightarrow & \quad a+b=0 \text{ or } c = 0 \text{ or } y+z \le 1. \label{eq:7} 
\end{align}

Observe that 
\begin{align}
    a &= 0 \text{ or } x+y \le 1, \label{eq:8} \\
    c &= 0 \text{ or } y+z \le 1. \label{eq:9}
\end{align}
To see \eqref{eq:8} if $b+c=0$ then by \eqref{eq:1} we have $a=n$ which according to \eqref{eq:2} implies $x+y \le 1$ so \eqref{eq:6} implies \eqref{eq:8}. Similarly, to see \eqref{eq:9} if $a+b=0$ then by \eqref{eq:1} we have $c=n$ which according to \eqref{eq:3} implies $y+z \le 1$ so \eqref{eq:7} implies \eqref{eq:9}.

To complete our analysis of Case 1, we show $a=c=0$, giving a contradiction to \eqref{eq:1}. Suppose to the contrary that $a\ge 1$. Then $x+y \le 1$ by \eqref{eq:8}, and so $z \ge n-1$ by \eqref{eq:1}, which according to \eqref{eq:4} shows $a+b \le 1$. Hence 
\[
c=n-(a+b)\ge n-1 \ge 1 \enskip \text{and} \enskip y+z \ge z \ge n-1 \ge 2,
\]
giving a contradiction to \eqref{eq:9}. It remains to show that $c=0$. If we instead have $c \ge 1$, then \eqref{eq:9} implies $y+z \le 1$, and so $x \ge n-1$ by \eqref{eq:1}, which by \eqref{eq:5} implies $b+c \le 1$. Thus
\[
a=n-(b+c)\ge n-1 \ge 1 \enskip \text{and} \enskip x+y \ge x \ge n-1 \ge 2,
\]
contradicting \eqref{eq:8} and completing the proof in this case.

\hspace*{0.3cm} {\bf Case 2:} $\vect{L}$ is of type (1,2).

The analysis of this case is very similar to that of Case 1. We first have the following observations
\begin{align}
{\rm (G1)} \Rightarrow & \quad a \le n-2 \text{ or } x+y \le 1,  \label{eq:10}\\
{\rm (G1)} \Rightarrow & \quad c \le n-2 \text{ or } y+z \le 1, \label{eq:11}\\
{\rm (H1)} \Rightarrow & \quad a+b \le 1 \text{ or } z \le n-1,  \label{eq:12}\\
{\rm (H1)} \Rightarrow & \quad b+c \le 1  \text{ or } x \le n-1,  \label{eq:13}\\
{\rm (P4)} \Rightarrow & \quad a+b \le 1 \text{ or } x+y=0 \text{ or } z =0,   \label{eq:14}\\
{\rm (P3)} \Rightarrow & \quad b+c \le 1 \text{ or } x = 0 \text{ or } y+z=0. \label{eq:15} 
\end{align}
We next show
\begin{align}
    a+b \le 1 & \text{ or } z = 0, \label{eq:16}\\
    b+c \le 1 & \text{ or } x = 0. \label{eq:17}
\end{align}

To see \eqref{eq:16} if $x+y=0$ then by \eqref{eq:1} we have $z=n$ which according to \eqref{eq:12} implies $a+b \le 1$ so \eqref{eq:14} implies \eqref{eq:16}. Similarly, to see \eqref{eq:17} if $y+z=0$ then by \eqref{eq:1} we have $x=n$ which according to \eqref{eq:13} implies $b+c \le 1$ so \eqref{eq:15} implies \eqref{eq:17}.

Finally, we show $x=z=0$, giving a contradiction to \eqref{eq:1}. Suppose $x\ge 1$. Then $b+c \le 1$ by \eqref{eq:17}, and so \eqref{eq:1} gives $a \ge n-1$, which by \eqref{eq:10} forces $x+y \le 1$. From this we conclude 
\[
a+b \ge a \ge n-1 \ge 2 \enskip \text{and} \enskip z=n-(x+y)\ge n-1 \ge 1  ,
\]
giving a contradiction to \eqref{eq:16}. To show $z=0$, we suppose $z \ge 1$. Then \eqref{eq:16} gives $a+b \le 1$, and so \eqref{eq:1} implies $c \ge n-1$, which by \eqref{eq:11} results in $y+z \le 1$. Thus
\[
b+c \ge c \ge n-1 \ge 2 \enskip \text{and} \enskip x=n-(y+z)\ge n-1 \ge 1,
\]
contradicting \eqref{eq:17}. This completes our proof of \Cref{thm:mon-to-lex} (i).
\end{proof}

The example used above is partially motivated by certain examples considered in \cite{piercing}. This paper also considers higher dimensional examples which may be of some use in higher dimensional instances of our problem as well, but only in terms of optimizing the dependency on $d$.

\subsection{High-dimensional case}
In this subsection we prove \Cref{thm:mon-to-lex} (ii). Our first ingredient in the proof will be the following lemma.

\begin{lem}[Dominant coordinate]
\label{lem:reduction}
Let $d,m,t \ge 2$ be integers, and let $f$ be an increasing array indexed by $[d^2mt]^d$. Then, there exist a dimension $i \in [d]$, sets $B_1,\ldots, B_{i-1},B_{i+1},\ldots, B_d \subseteq [d^2mt]$ of size $m+1$ and $t$ subgrids $\vect{A}_h:=B_1 \times\ldots \times B_{i-1} \times \{h\} \times B_{i+1} \times \ldots \times B_d$ such that
$f\rvert_{\vect{A}_h}< f\rvert_{\vect{A}_{h'}}$ whenever $h<h'$. 
\end{lem}

One should think of this lemma as saying that there is a dimension $i$ such that one can find a ``stack'' of subgrids appearing at the same location along the remaining $d-1$ dimensions and different positions along dimension $i$, which can be thought of as heights of the subgrids. Furthermore, the subgrids are not much smaller than the initial one in the remaining dimensions and our array is always bigger on a higher subgrid. Our proof of this lemma borrows some ideas of \cite{FG93}.
\begin{proof}
The proof of the lemma is in some sense a high-dimensional generalisation of the argument used to prove \Cref{thm:F-G}. We split the grid $[d^2mt]^d$ along each coordinate into $td$ intervals of equal size, obtaining a partition of $[d^2mt]^d$ into translates of $[dm]^d$
\begin{equation}\label{eq:translates}
[d^2mt]^d=\bigcup_{\vect{u}\in \vect{T}}\left(\vect{u}+[dm]^d\right),    
\end{equation}
where $\vect{T}=\{0,dm,2dm,\ldots,(dt-1)dm\}^d$. The reason behind considering this is that we are now going to compare values taken by the array on certain points in  
$\vect{u}+[dm]^d$ for each $\vect{u} \in \vect{T}$, and once we find the one with the largest entry, the fact that points of $\vect{T}$ are suitably spaced apart will allow us to get information about the ordering of a relatively large $(d-1)$-dimensional subarray.
For each $i \in [d]$, the aforementioned points are given by $\vect{x}_i$ and the subarrays by $\vect{C}_i$ below.
\begin{align*}
\vect{x}_i& =((i-1)m, (i-2)m,\ldots,m,dm,(d-1)m,\ldots,im) \in [dm]^d,\\    
\vect{C}_i& =[(i-1)m,im]\times \ldots \times [m,2m] \times \{m\} \times [(d-1)m,dm]\times \ldots \times [im,(i+1)m] \subset [dm]^d.    
\end{align*}
Notice that $\vect{x}_{i+1}$ has every coordinate larger than $\vect{x}_{i}$, except $i$-th, here $\vect{x}_{d+1}:=\vect{x}_1$. Notice further that $\vect{x}_{i+1} \in \vect{C}_i$ is larger in every coordinate than any other point of $\vect{C}_i$. Similarly $\vect{x_i}$ with its $i$-th coordinate reduced by $(d-1)m$ is the point of $C_i$ which is smaller than any other in every coordinate. In other words with respect to the componentwise order of $[dm]^d$:
\begin{equation}
\label{eq:min-max}
\max \vect{C}_i=\vect{x_{i+1}}, \enskip \text{and} \enskip \min \vect{C}_{i}=\vect{x_i}-(d-1)m\vect{e}_i,    
\end{equation}
where $\vect{e}_i$ stands for the $i$-th unit vector $(0,\ldots, \underset{i\text{-th}}{\underset{\uparrow}{1}}, \ldots,0)$ .

\begin{figure}[h]
\begin{minipage}{0.33\textwidth}
    \centering
    \begin{tikzpicture}[scale=0.55]

\defPt{0}{0}{x1}
\defPt{5}{0}{x2}
\defPt{0}{5}{y1}
\defPt{5}{5}{y2}

\draw[line width=1 pt, blue!30] (x1) -- (x2);
\draw[line width=1 pt, blue] (y1) -- (y2);

\draw[line width=1 pt, red!30] (x1) -- (y1);
\draw[line width=1 pt, red] (x2) -- (y2);

    \draw[] (x2) \bvx;
    \draw[] (y1) \bvx;
        \node[] at ($(x2)+(0,-0.7)$) {{\textbf{$\vect{x_1}=(2m,m)$}}};
        \node[] at ($(y1)+(0,0.7)$) {{\textbf{$\vect{x_2}=(m,2m)$}}};

\node[] at (0,-0.7) {{$(m,m)$}};
\node[] at (5,5.7) {{$(2m,2m)$}};

\end{tikzpicture}
\end{minipage}
\begin{minipage}{0.66\textwidth}
    \centering
    \begin{tikzpicture}[scale=0.55]

\defPt{-8.13}{-3.15}{x1}
\defPt{-2.8229}{-4.5719}{x2}
\defPt{-1.73}{-1.564}{x3}
\defPt{3.574}{-2.985}{x4}

\defPt{-8.13}{3.94}{y1}
\defPt{-2.8229}{2.52}{y2}
\defPt{-1.734}{5.53}{y3}
\defPt{3.574}{4.10578}{y4}

\draw[line width=1 pt] (x1) -- (x2) -- (x4);

\draw[line width=1 pt, dotted] (x4) -- (x3) -- (x1);

\draw[line width=1 pt] (y1) -- (y2) -- (y4) -- (y3) -- (y1);

    \path [fill=blue!70, opacity=.5] ($0.5*(x4)+0.5*(x2)$) to  ($0.5*(x4)+0.5*(y2)$) to ($0.5*(x4)+0.5*(y4)$) to (x4) to ($0.5*(x4)+0.5*(x2)$);

    \path [fill=blue!30, opacity=.5] ($0.5*(x3)+0.5*(x1)$) to  ($0.5*(x3)+0.5*(y1)$) to ($0.5*(x3)+0.5*(y3)$) to (x3) to ($0.5*(x3)+0.5*(x1)$);
    
    \path [fill=black!60!green, opacity=0.5] ($0.5*(y3)+0.5*(x3)$) to  ($0.5*(x3)+0.5*(y4)$) to ($0.5*(y3)+0.5*(y4)$) to (y3) to ($0.5*(y3)+0.5*(x3)$);

    \path [fill=green!20, opacity=.7] ($0.5*(y1)+0.5*(x1)$) to  ($0.5*(x1)+0.5*(y2)$) to ($0.5*(y1)+0.5*(y2)$) to (y1) to ($0.5*(y1)+0.5*(x1)$);
    
    \path [fill=red!30, opacity=.5] ($0.5*(x1)+0.5*(x2)$) to  ($0.5*(x2)+0.5*(x3)$) to ($0.5*(x4)+0.5*(x2)$) to (x2) to ($0.5*(x1)+0.5*(x2)$);

    \path [fill=black!10!red!90, opacity=.5] ($0.5*(y1)+0.5*(y2)$) to  ($0.5*(y2)+0.5*(y3)$) to ($0.5*(y4)+0.5*(y2)$) to (y2) to ($0.5*(y1)+0.5*(y2)$);
    
    \foreach \j in {1,2,4}
    {
        \draw[line width=1 pt] (x\j) -- (y\j);
    }
    \draw[line width=1 pt,dotted] (x3) -- (y3);
    
    \draw[line width=1 pt]
    ($0.5*(x4)+0.5*(x2)$) to  ($0.5*(x4)+0.5*(y2)$) to ($0.5*(x4)+0.5*(y4)$);
    
    \draw[line width=1 pt, dotted]
    ($0.5*(x3)+0.5*(x1)$) to  ($0.5*(x3)+0.5*(y1)$) to ($0.5*(x3)+0.5*(y3)$);
    
    \draw[line width=1 pt, dotted]
    ($0.5*(y3)+0.5*(x3)$) to  ($0.5*(x3)+0.5*(y4)$) to ($0.5*(y3)+0.5*(y4)$);
    
    \draw[line width=1 pt]
    ($0.5*(y1)+0.5*(x1)$) to  ($0.5*(x1)+0.5*(y2)$) to ($0.5*(y1)+0.5*(y2)$);
    
    \draw[line width=1 pt, dotted]
    ($0.5*(x1)+0.5*(x2)$) to  ($0.5*(x2)+0.5*(x3)$) to ($0.5*(x4)+0.5*(x2)$);
    
    \draw[line width=1 pt]
    ($0.5*(y1)+0.5*(y2)$) to  ($0.5*(y2)+0.5*(y3)$) to ($0.5*(y4)+0.5*(y2)$);
    
    \draw[] ($0.5*(x2)+0.5*(x4)$) \bvx;
        \node[] at ($0.5*(x2)+0.5*(x4)-(-0.2,0.6)$) {{\textbf{$\vect{x_1}$}}};
    \draw[] ($0.5*(y1)+0.5*(y2)$) \bvx;
        \node[] at ($0.5*(y1)+0.5*(y2)-(-0.5,0.5)$) {{\textbf{$\vect{x_3}$}}};
    \draw[] ($0.5*(x3)+0.5*(y3)$) \bvx;
        \node[] at ($0.5*(x3)+0.5*(y3)+(0.5,-0.5)$) {{\textbf{$\vect{x_2}$}}};
\node[] at ($(x1)+(0,-0.75)$) {{$(m,m,m)$}};
\node[] at (7.5,2) {{$\vect{x_1}=(3m,2m,m)$}};
\node[] at (7.5,1) {{$\vect{x_2}=(m,3m,2m)$}};
\node[] at (7.5,0) {{$\vect{x_3}=(2m,m,3m)$}};
\node[] at ($(y4)+(1.5,0.6)$) {{$(3m,3m,3m)$}};   
\end{tikzpicture}
    \caption{Lightly shaded $(d-1)$-dimensional regions in the figure denote $\vect{C_i}$'s with their maximum point being $\vect{x_{i+1}}$. Depending on which of the $\vect{x_i}$'s has largest value of $f$ one of these $\vect{C_i}$'s has value of $f$ on $\vect{x_i}$ smaller than that of $f$ on the minimal point of a translate of $\vect{C_i}$ (strongly shaded region of the same colour) at $\vect{x_{i}}$.}
    \label{fig:ap}
\end{minipage}
\end{figure}
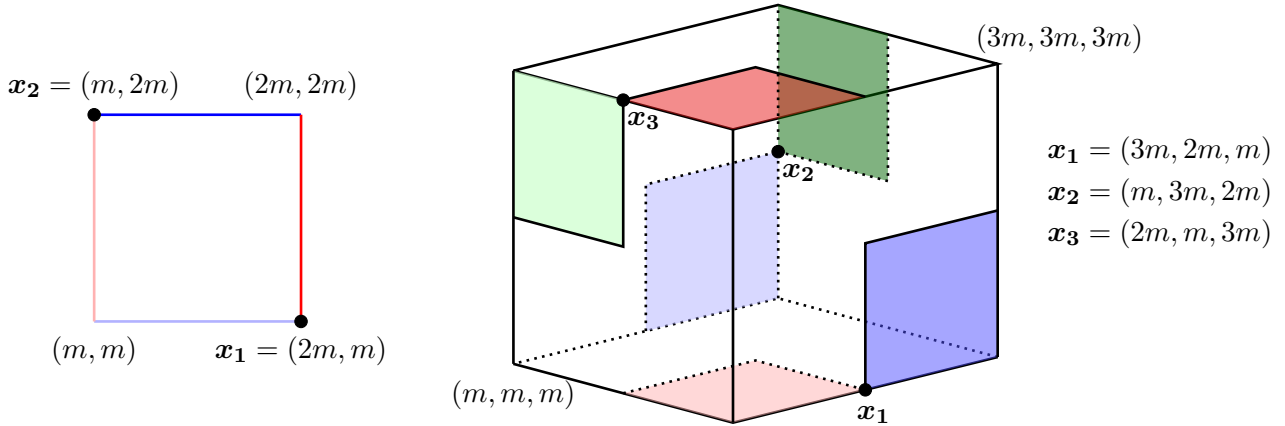

Now consider a colouring $\chi\colon \vect{T} \rightarrow [d]$ given by: 
\[
\chi(\vect{u})=i \enskip \text{if and only if} \enskip f(\vect{u}+\vect{x}_i)=\max \{ f(\vect{u}+\vect{x}_1),\ldots,f(\vect{u}+\vect{x}_d)\}.
\]
By pigeonhole principle, there is a colour $i\in [d]$ which appears at least $(td)^d/d$ times. This implies that the grid $\vect{T}$ contains a column in the direction of the $i$-th coordinate for which at least $t$ vertices of this column have colour $i$. We list those  vertices of $\vect{T}$ from smallest to largest with respect to their $i$-th coordinates: $\vect{u}_1,\ldots,\vect{u}_t$.

We show that the grids $\vect{A}_1=\vect{u}_1+\vect{C}_i,\ldots,\vect{A}_t=\vect{u}_t+\vect{C}_i$ have the desired properties. Indeed, \eqref{eq:translates} implies that $\vect{A}_1,\ldots,\vect{A}_t$ are subgrids of $[d^2mt]^t$. Since we have chosen $\vect{u}_j$'s as in the same column in the direction of the $i$-th coordinate, all of them have the same coordinates in all other dimensions. This implies that there are $d-1$ sets $B_1,\ldots,B_{i-1},B_{i+1},\ldots,B_d$ and $t$ ``heights'' $h_1,\ldots,h_t$ such that $\vect{A}_j$ can be written as $B_1 \times\ldots \times B_{i-1} \times \{h_j\} \times B_{i+1} \times \ldots \times B_d$ for every $1\le j \le t$. Since $\vect{C}_i$ has size $(m+1) \times \ldots \times (m+1)$ each $B_j$ has size $m+1.$ Finally, since $f$ is increasing, for $1\le j<k \le t$ we have 
\begin{align*}
\max f(\vect{A}_j)&=f(\vect{u}_j+\vect{x}_{i+1})\\
 &<f(\vect{u}_j+\vect{x}_i) \\
 &<f(\vect{u}_k-(d-1)m\vect{e}_i+\vect{x}_i)\\
&=\min f(\vect{A}_k).
\end{align*}
The first equality follows since \eqref{eq:min-max} implies $\vect{u}_j+\vect{x}_{i+1}$ is the largest point of $\vect{A}_j$ so since $f$ is increasing we conclude that $f$ is maximised over $\vect{A}_j$ at $\vect{u}_j+\vect{x}_{i+1}$. Similarly, we get the last equality as well.
The first inequality follows since $\chi(\vect{u}_j)=i$. The second inequality follows since $j<k$ implies 
the $i$-th coordinate of $\vect{u}_j$ is smaller than that of $\vect{u}_k$ (since 
$\vect{u}_j$ and $\vect{u}_k$ belong to the same column along $i$-th dimension and since we named them according to their $i$-th coordinate) by at least $dm$ (since $\vect{u}_j,\vect{u}_k\in \vect{T})$. This finishes our proof of \Cref{lem:reduction}.
\end{proof}
\vspace{-0.5cm}
This lemma provides us with a stack of subgrids $\vect{A}_1,\ldots, \vect{A}_t$ on which $f$ is increasing in dimension $i$. It is natural to try to iterate and now apply the lemma within each $\vect{A}_j$, but notice that we do not only want to find a lexicographic subgrid in some $n$ different $\vect{A}_j$'s, but they also have to appear at the same positions in each of them. One can now apply a Ramsey result for $O(2^{\ML_{d-1}(n)})$ colours to ensure the found subgrids appear at the same locations. Consequently, this approach gives at best $\ML_d(n) \le \towr_{d-1}(O_d(n))$. Fishburn and Graham \cite[Section 4]{FG93} follow a similar approach and their arguments hit the same barrier in terms of the bounds they can obtain.

We take a different approach in order to prove \Cref{thm:mon-to-lex} (ii).
\begin{thm}
For $d\ge 3$, we have $\ML_d(n)\le 2^{(c_d+o(1))n^{d-2}}$, where $c_d=\frac{1}{2}(d-1)!$.
\end{thm}
\begin{proof}
We prove the statement by induction on $d$.

\hspace{0.3 cm} {\bf The base case:} $d=3$.

Let $m=2n^3$, $t=2m \binom{m/n}{n}/\binom{m/(2n)}{n}=2^{n+o(n)}$, and $N=d^2mt=2^{n+o(n)}$. Consider an increasing array $f:[N]^3\rightarrow \bR$. By \Cref{lem:reduction}, we can assume w.l.o.g. that $[N]^3$ contains a stack of $t$ subgrids $\vect{A}_1=  B_1\times B_2 \times \{h_1\}, \ldots,\vect{A}_t=B_1\times B_2\times \{h_t\}$ of size $m\times m \times 1$ such that $h_1<h_2<\ldots<h_t$ and $f\rvert_{\vect{A}_1}<f\rvert_{\vect{A}_2}<\ldots<f\rvert_{\vect{A}_t}$. We drop the third dimension from now on, and think of $\vect{A}_i$'s as 2-dimensional grids.

We split each $\vect{A}_i$ into $\left(\frac{m}{n}\right)^2$ smaller subgrids of size $n \times n$. Colour each such smaller subgrid red if its topmost leftmost corner is smaller than bottommost rightmost, and blue otherwise.
As in the proof of \Cref{thm:F-G}, any $n$ red subgrids in the same row of $\vect{A}_i$ give rise to an $n \times n$ subgrid of type $(1,2)$. If we further manage to find $n$ layers of the stack, each having such a sequence of the same $n$ red $n\times n$ subgrids  we obtain an $n\times n \times n$ subgrid of type $(3,1,2)$ (using the property that $h_1<h_2<\ldots<h_t$ and $f\rvert_{\vect{A}_1}<f\rvert_{\vect{A}_2}<\ldots<f\rvert_{\vect{A}_t}$). Similarly, if we find $n$ layers each having the same sequence of $n$ blue $n \times n$ subgrids in the same column we find an $n \times n \times n$ subgrid of type $(3,2,1)$. 

 By pigeonhole principle, each layer $\vect{A}_i$ of the stack has a row with $\frac{m}{2n}$ red subgrids or a column with at least $\frac{m}{2n}$ blue subgrids. This row or column can be chosen in $2\cdot \frac{m}{n}$ ways, so there are $\frac{nt}{2m}$ layers having in the same row/column $\frac{m}{2n}$ red/blue subgrids. Let's say w.l.o.g. that it is the first row.  Then, the first row of such a layer contains at least $\binom{m/(2n)}{n}$ tuples of $n$ red subgrids. By pigeonhole principle, there are at least 
 \[
 \frac{\frac{nt}{2m}\binom{m/(2n)}{n}}{\binom{m/n}{n}} \ge n
 \]
 layers having the same red $n$-tuples, as desired.

\hspace{0.3 cm} {\bf The induction step:} suppose $d\ge 4$ and that the lemma holds for $d-1$.

It is easy to see that the desired estimate $\ML_d(n) \le 2^{(c_d+o(1))n^{d-2}}$ follows from the induction hypothesis and the following recursive bound 
\[
\ML_d(n) \le d^{d}\ML_{d-1}(n)^{(d-1)n+1} \enskip \text{for every} \enskip d\ge 4 \enskip \text{and} \enskip n \ge 2. 
\]
We now prove this inequality. Let $m=\ML_{d-1}(n)$, $t=(d-1)!n\binom{m}{n}^{d-1}$, and $N=d^2mt \le d^d\ML_{d-1}(n)^{(d-1)n+1}$. 
Consider an increasing array $f:[N]^d\rightarrow \bR$. 

By \Cref{lem:reduction}, we can assume w.l.o.g. that $[N]^d$ contains a stack of $t$ subgrids $\vect{A}_1=  B_1\times \ldots \times B_{d-1}\times \{h_1\}, \ldots,\vect{A}_t=B_1\times \ldots \times B_{d-1}\times \{h_t\}$ of size $m\times \ldots \times m \times 1$ such that $h_1<h_2<\ldots<h_t$ and $f\rvert_{\vect{A}_1}<f\rvert_{\vect{A}_2}<\ldots<f\rvert_{\vect{A}_t}$.

Given $i\in [t]$, as $m=\ML_{d-1}(n)$, one can find a permutation $\sigma \in \mathfrak{S}_{d-1}$ and a subgrid $\vect{A}'_i \subset \vect{A}_i$ of size $n\times \ldots \times n \times 1$ such that $f\rvert_{\vect{A}'_i}$ is of type $\sigma$. Since $\frac{t}{(d-1)!\binom{m}{n}^{d-1}}=n$, the pigeonhole principle implies the existence of a permutation $\sigma \in \mathfrak{S}_{d-1}$, an $n\times \ldots \times n$ subgrid $B'_1\times \ldots \times B'_d$ of $B_1\times \ldots \times B_d$, and $n$ layers $1\le i_1<\ldots<i_n \le t$ such that for every $k\in [n]$, we have that $\vect{A}'_{i_k}=B'_1\times \ldots \times B'_{d-1}\times \{h_{i_k}\}$, and that the restriction of $f$ to $\vect{A}'_{i_k}$ is of type $\sigma$. As $f\rvert_{\vect{A}_{i_1}}<\ldots < f\rvert_{\vect{A}_{i_n}}$,
the restriction of $f$ to $B'_1\times \ldots \times B'_{d-1}\times \{h_{i_1},\ldots,h_{i_n}\}$ is an $n\times \ldots \times n$ array of type $(d,\sigma)$. This shows $\ML_d(n) \le N \le d^d\ML_{d-1}(n)^{(d-1)n+1}$, as required.
\end{proof}

\section{Concluding Remarks}
We obtain a major improvement on best known upper bounds for $M_d(n)$. However, our bounds are still off from the best known lower bound of $M_d(n) \ge n^{(1+o(1))n^{d-1}/d}$ due to Fisburn and Graham \cite[Theorem 3]{FG93}.
Perhaps the most interesting open question regarding $M_d(n)$ is to determine the behaviour in 2 dimensions.
\begin{qn}\label{qn1}
What is the behaviour of $M_2(n)$? Is it closer to exponential or to double exponential in $n$?
\end{qn}

It is natural to ask whether our argument used to get 
a double exponential bound in the monotone case in 3 dimensions (\Cref{thm:monotone-3d}) generalises to higher dimensions. Unfortunately, the natural generalisation of our approach to more dimensions gives a bound of the form $M_d(n) \le \towr_{\floor{d/2}+2}(O_d(n))$ which has a tower of height growing with $d$. However, this does still imply a better bound than \Cref{thm:monotone-anyd} in 4 and 5 dimensions. 
The main issue preventing us from extending our argument to more dimensions is the fact that it seems hard to obtain asymmetric results which would allow us to find a monotone subarray with exponential size in at least $2$ dimensions. For example, if we could find an $n \times 2^n \times 2^{n^2}$ monotone subarray within any array of size $2^{2^{O(n^2)}}\times 2^{2^{O(n^2)}} \times 2^{2^{O(n^2)}}$ we would obtain a double exponential bound
$M_4(n) \le 2^{2^{O(n^3)}}$. However, if we knew how to do this then by considering an array which is always increasing in the first dimension and has the same but arbitrary ordering for each 2-dimensional subarray with fixed value in the first dimension, we would also be able to get a better than double exponential bound in the 2-dimensional case, which leads us back to \Cref{qn1}.
Our better bounds in 3, 4 and 5 dimensions make it seem unlikely that a triple exponential is ever needed.

\begin{qn}
For $d \ge 4$ is $M_d(n)$ bounded from above by a double exponential in $n^{d-1}$?
\end{qn}

For the problem of determining $\ML_d(n)$, we completely settle the 2-dimensional case and give exponential upper bounds for $d \ge 3$. The best known lower bound $\ML_d(n) \ge (n-1)^d$, also due to Fishburn and Graham \cite{FG93} is still only polynomial. We find the 3-dimensional case particularly interesting since via \Cref{lem:reduction} it reduces to the following nice problem. 
\begin{qn} What is the smallest $N$ such that given $N$ increasing arrays of size $N \times N$ one can find an $n\times n$ lexicographic array of the same type appearing in the same positions in at least $n$ of the arrays?
\end{qn}
In particular, is this $N$ bounded by a polynomial in $n$ or is it exponential in $n$.

The study of $M_d(n),\ML_d(n)$ and $L_d(n)$ while interesting in its own right is also closely related to various other interesting problems. We present just a few here.
\subsection{Long common monotone subsequence}
The problem of estimating $M_2(n)$ is closely related to the longest common monotone subsequence problem. A \textit{common monotone subsequence} of two permutations $\pi, \sigma \in \mathfrak{S}_N$ is a set $I \subseteq [N]$ such that the restrictions of  $\pi$ and $\sigma$ to $I$ are either both increasing or both decreasing. A common monotone subsequence of more than two permutations is defined analogously. Given positive integers $t,k$ and $N$, let $\hbox{LMS}(t,k,N)$ denote the maximum $\ell$ such that any size-$k$ multisubset $P\subseteq \mathfrak{S}_N$ contains a size-$t$ multisubset $P'$ such that the length of the longest common monotone subsequence of $P'$ is at least $\ell$. 

We now describe the connection between $M_2(n)$ and $\LMS(t,k,N)$. Let $f\colon [N]^2\rightarrow \bR$ be a 2-dimensional array. Similarly to the first part of the proof of \Cref{thm:monotone-asym-2d}, we can show that $[N]^2$ contains a subgrid $R\times C$ of size $(\log N)^{1-o(1)}\times N^{1-o(1)}$ such that either $f\rvert_{R\times C}$ is increasing in each column, or $f\rvert_{R\times C}$ is decreasing in each column. For each $r\in R$, the restriction of $f$ to the row $\{r\}\times C$ induces a permutation $\pi_r$ of $C$, since $f$ is assumed to be injective. (Note that $\pi_r$'s are not necessarily distinct.) It is not hard to see that if among $(\log N)^{1-o(1)}$ permutations $\{\pi_r:r\in R\}$ of $C$ there are $n$ permutations whose longest common mononotone subsequence has length at least $n$, then $f\rvert_{R\times C}$ contains an $n\times n$ monotone subarray. Therefore every $N\times N$ array contains a monotone subarray of size $n\times n$, where $n$ is the maximum $t\in \N$ such that $\LMS(t,(\log N)^{1-o(1)},N^{1-o(1)})$ is greater than or equal to $t$. By an iterative application of the Erd\H{o}s-Szekeres theorem, one can take $n= (1/2-o(1))\log_2 \log_2 N$, or equivalently $N=2^{2^{(2+o(1))n}}$.   

The problem of determining the parameter $\LMS(t,k,N)$ for other ranges of $t$ and $k$ is also very appealing. For example, it would be interesting to have a good estimate for $\LMS(t,k,N)$ when $t$ is fixed and $k$ grows to infinity with $N$. We refer the reader to \cite{BBN09,BN08,BZ16} for some related results in this direction. 

\subsection{Ramsey type problems for vertex-ordered graphs}

One can place the problems we have studied in this paper under the framework of (vertex-)ordered Ramsey numbers. For simplicity of presentation we choose to illustrate this through the 2-dimensional monotone subarray problem. Let $K^{(3)}_{N,N}$ be the $3$-uniform hypergraph with vertex set $A\cup B$, where $A$ and $B$ are two copies of $[N]$, and edge set consisting of all those triples which intersect both $A$ and $B$. 
Given an array $f\colon [N]^2\rightarrow \bR$, we can associate to $f$ an edge-colouring $\chi$ of $K^{(3)}_{N,N}$ with two colours red and blue. For $i\in A$ and $j,j'\in B$ with $j<j'$, let $\chi(i,j,j')=\text{red}$ if and only if $f(i,j)<f(i,j')$. Similarly, for $j \in B$ and $i,i' \in A$ with $i<i'$, we assign colour red to $(i,i',j)$ if and only if $f(i,j)<f(i',j)$. The following simple observation connects the multidimensional Erd\H{o}s-Szekeres world with the ordered Ramsey world.
\begin{obs}
Suppose there are two size-$n$ subsets $\{a_1< \ldots < a_n\} \subseteq A$, $\{b_1<\ldots<b_n\} \subseteq B$ s.t. \begin{itemize}
\item $\{(a_i,b_j,b_{j+1}): i \in [n], j \in [n-1]\}$ is monochromatic under $\chi$,
\item $\{(a_i,a_{i+1},b_j):j \in [n], i \in [n-1]\}$ is monochromatic under $\chi$.
\end{itemize}
Then the restriction of $f$ to $\{a_1,\ldots,a_n\} \times \{b_1,\ldots,b_n\}$ is monotone.
\end{obs}

Now define $\hbox{OR}(n)$ to be the smallest $N$ such that in every red-blue colouring of the edges of $K^{(3)}_{N,N}$ we can always find two size-$n$ subsets $\{a_1< \ldots < a_n\} \subseteq A$ and $\{b_1<\ldots<b_n\} \subseteq B$ with the aforementioned properties. From the observation, we know $M_2(n) \le \hbox{OR}(n)$. A closer inspection of our proof of the inequality $M_2(n) \le 2^{2^{(2+o(1))n}}$ reveals that it actually gives $\hbox{OR}(n) \le 2^{2^{(2+o(1))n}}$. Thus it is natural to ask whether $M_2(n)$ and $\hbox{OR}(n)$ have the same order of magnitude.

\subsection{Canonical orderings of discrete structures}

An ordering of the edges of a (vertex-ordered) $d$-graph $G$ with $V(G) \subset \bZ$ is {\em lex-monotone} if one can find a permutation $\sigma \in \mathfrak{S}_d$ and a sign vector $\vect{s} \in \{-1,1\}^d$ such that the edges $(a_1,\ldots,a_d)$ of $G$ with $a_1<\ldots<a_d$ are ordered according to the lexicographical order of the tuple $(s_{\sigma(1)}a_{\sigma(1)},\ldots, s_{\sigma(d)}a_{\sigma(d)})$. An old result of Leeb and Pr{\"o}mel (see \cite[Theorem 2.8]{Nes84}) says that for every $d,n\in \N$ there is a positive integer $\LP_d(n)$ such that every 
edge-ordering of a (vertex-ordered) complete $d$-graph on $\LP_d(n)$ vertices contains a copy of the complete $d$-graph on $n$ vertices whose edges induce a lex-monotone ordering. \Cref{thm:lexicographic} in our paper can be viewed naturally as a $d$-partite version (with a better bound) of this result. It would be interesting to know if our approach can lead to an improvement on the upper bound $\LP_d(n) \le \towr_{2d}(O_d(n))$ for $d\ge 2$, due to Ne\v{s}et\v{r}il and R{\"o}dl \cite[Theorem 14]{NR17}. For other interesting results in edge-ordered Ramsey numbers, we refer the reader to \cite{BV19,FL19}.

\Cref{thm:lexicographic} is also related to the work of Ne\v{s}et\v{r}il, Pr{\"o}mel, R{\"o}dl and B. Voigt \cite{NPRV85} on linear orders of the combinatorial cube $[k]^n$ when $k$ is fixed and $n$ is large. For simplified presentations of this work, see \cite{Pro88,BS19}.


\end{document}